\documentclass{article}

\usepackage[left= 1in, right = 1in, top = 1in, bottom = 1in]{geometry}

\usepackage{amsmath, amsfonts, amssymb, amsthm, mathtools, thmtools}

\usepackage{bm}

\usepackage{url}

\usepackage{enumerate}

\usepackage[hidelinks]{hyperref}

\usepackage[capitalize]{cleveref}

\usepackage[T1]{fontenc}
\usepackage{graphicx}

\usepackage{mathtools}

\usepackage{tikz}
\usetikzlibrary{calc}
\usetikzlibrary{decorations.pathreplacing}
\usetikzlibrary{decorations.pathmorphing}
\usetikzlibrary{shapes.geometric}
\usetikzlibrary{fit}
\usetikzlibrary{arrows.meta}
\usetikzlibrary{arrows}

\usepackage{todonotes}

\DeclareMathOperator{\boxicity}{box}

\DeclareMathOperator{\thdim}{dim_{TH}}
\DeclareMathOperator{\Min}{Min}

\newcommand{\Iscr}{\mathcal{I}}
\newcommand{\Fscr}{\mathcal{F}}

\newtheorem{theorem}{Theorem}
\newtheorem{lemma}[theorem]{Lemma}
\newtheorem{proposition}[theorem]{Proposition}
\newtheorem{corollary}[theorem]{Corollary}
\newtheorem{claim}{Claim}

\theoremstyle{definition}
\newtheorem{definition}[theorem]{Definition}
\newtheorem{remark}[theorem]{Remark}
\newtheorem{question}[theorem]{Question}

\crefname{lemma}{Lemma}{Lemmas}
\crefname{proposition}{Proposition}{Propositions}
\crefname{corollary}{Corollary}{Corollaries}
\crefname{definition}{Definition}{Definitions}
\crefname{remark}{Remark}{Remarks}
\crefname{question}{Question}{Questions}

\title{Boxicity and Threshold Dimension of Zero Divisor Graphs}

\author{Marco Caoduro and Meike Neuwohner}

\date{}

\begin{document}

	\maketitle
	\begin{abstract}
	The \emph{zero divisor graph} $\Gamma(R)$ of a finite commutative ring $R$ has as vertices the non-zero zero divisors of $R$, with an edge between two elements exactly when their product is zero. 
	We determine the boxicity and threshold dimension of $\Gamma(R)$ for two classes of finite commutative rings: reduced rings and quotients of principal ideal domains.
	Our proofs use a new combinatorial gadget, the \emph{integral covering graph}, that captures the structure shared by both ring families and generalizes the \emph{disjointness graph} on subsets of $[n]$, where two subsets are adjacent if and only if they are disjoint.
	In doing so, we answer two questions recently posed by L.~Sunil Chandran and Suraj Kumar Sahoo in \emph{Boxicity of Zero Divisor Graphs}, Discrete Applied Mathematics \textbf{391} (2026).
	\end{abstract}
	%

	\section{Introduction}
	The \emph{intersection graph} of a finite set family $\Fscr$ is the graph with vertex set $\{v_A : A \in \Fscr\}$ and edge set $\{v_A v_B : A, B \in \Fscr,\ A \cap B \neq \emptyset\}$.
	\begin{definition}
		The \emph{boxicity} of a graph $G$, denoted by $\boxicity(G)$, is the minimum integer $d$ such that $G$ admits a representation as the intersection graph of a family of $d$-dimensional axis-parallel boxes in $\mathbb{R}^d$.
	\end{definition}
	Boxicity was introduced by Roberts in 1969~\cite{1969_Roberts}.
	The original motivation for its study arose from applications in ecology, where intersection models were used to describe niche overlap and competition between species~\cite{1978_Cohen,1976_Roberts}.
	Subsequently, boxicity has found applications in operations research, notably in fleet maintenance and task assignment problems~\cite{1983_Cozzens,1981_Opsut}.
	Boxicity is by now a well-studied graph parameter and has been investigated in relation to several classical graph invariants.
	Upper bounds on boxicity are known in terms of parameters such as the number of vertices~\cite{1969_Roberts}, maximum degree~\cite{2008_Chandran,2020_Scott}, vertex cover number~\cite{2009_Chandran}, acyclic colouring number~\cite{2013_Esperet}, and treewidth~\cite{2007_Chandran}.
	In addition, the boxicity of several important graph classes has been determined or bounded.
	These include outerplanar graphs~\cite{1984_Scheinerman}, planar graphs~\cite{1986_Thomassen}, bipartite graphs~\cite{1993_Bellantoni,1991_Hartman}, and line graphs~\cite{2011_Chandran}.
	For a comprehensive overview of results and techniques related to boxicity, we refer the reader to the recent survey~\cite{2026_Chandran}.
	
	Besides its geometric definition, boxicity can also be defined as a purely graph-theoretic parameter.
	The following classical result is due to Roberts~\cite{1969_Roberts}.
	\begin{lemma}\label{lemma:boxicity_interval_graphs}
		A graph $G$ has boxicity at most $d$ if and only if there exist $d$ interval graphs $I_1,\ldots,I_d$ on the same vertex set as $G$ such that
		\[
		E(G) = \bigcap_{j=1}^d E(I_j).
		\]
		Equivalently, each $I_j$ is a supergraph of $G$ and $G$ is the edge-wise intersection of these interval graphs.
	\end{lemma}
	Given this definition, it is natural to replace the class of interval graphs with other graph classes.
	These `intersection dimension' parameters have been studied by Kratochv\'il and Tuza~\cite{1994_Kratochvil}.
	A particularly interesting parameter is obtained by replacing the class of interval graphs with that of threshold graphs.
	
	\begin{definition}
		Let $V$ be a finite set, let $\alpha\colon V\rightarrow\mathbb{R}$ and let $S\in\mathbb{R}$. We call the graph $T(\alpha,S)\coloneqq (V,\{\{v,w\}\in \binom{V}{2}\colon \alpha(v)+\alpha(w)\ge S\})$ the \emph{threshold graph induced by $\alpha$ and $S$}.
		We call a graph $G=(V,E)$ a \emph{threshold graph} if $G=T(\alpha,S)$ for some $\alpha$ and $S$.
	\end{definition}
	
	\begin{definition}
		The \emph{threshold dimension} of a graph $G$, denoted by $\thdim(G)$, is the smallest $d$ such that there exist $d$ threshold graphs $H_1,\ldots,H_d$ on the same vertex set as $G$ such that
		\[
		E(G) = \bigcap_{j=1}^d E(H_j).
		\]
	\end{definition}
	Note that the class of threshold graphs is contained in the class of interval graphs, hence $\boxicity(G) \leq \thdim(G)$ for any graph $G$. 
	While there are interval graphs with arbitrarily high threshold dimension (see \cite[Proposition 4]{2025_Francis}), the ratio between the threshold dimension and the boxicity of a graph is bounded by its chromatic number (see \cite[Theorem 19]{2021_Chacko}).

	\subsection{Reduced rings}
	In our work, we study the boxicity of the zero divisor graphs of rings, building on recent work in~\cite{2025_Chandran_ZeroDivisor,2025_Kavaskar}.
	Let $R$ be a finite commutative ring with non-zero identity.
	We call $r\in R$ a \emph{zero divisor} if there exists $s\in R\setminus \{0\}$ such that $r\cdot s=0$.
	We denote the set of zero divisors of $R$ by $Z(R)$.
	The object we are interested in this paper is the \emph{zero divisor graph} of a ring.
	\begin{definition}
		The \emph{zero divisor graph} $\Gamma(R)$ is the simple graph on the vertex set $Z(R)\setminus\{0\}$ in which two elements $r,s\in Z(R)\setminus \{0\}$ are connected by an edge if and only if $r\cdot s=0$.
	\end{definition}
	
	First, we address zero divisor graphs of \emph{reduced rings}.
	\begin{definition}
		We call $R$ \emph{reduced} if for any $r \in R \setminus \{0\}$ and $n\ge 1$, $r^n \neq 0$.
	\end{definition}

	Kavaskar~\cite{2025_Kavaskar} initiated the study of the boxicity of zero divisor graphs by studying reduced rings.
	Their results were subsequently strengthened by Chandran and Sahoo~\cite{2025_Chandran_ZeroDivisor}.
	\begin{theorem}
		Let $R$ be a finite commutative reduced ring with non-zero identity. 
		Then
		\[
		\frac{n}{2} \leq \boxicity(\Gamma(R)) \leq \thdim(\Gamma(R)) \leq n,
		\]
		where $n$ denotes the number of minimal prime ideals of $R$, or equivalently, the chromatic number of $\Gamma(R)$.
	\end{theorem}
	Chandran and Sahoo~\cite{2025_Chandran_ZeroDivisor} asked whether the lower bound is tight~\cite[Question 5.3]{2025_Chandran_ZeroDivisor}.

	Our first main result characterizes the boxicity of this class of zero divisors, giving a negative answer to Chandran and Sahoo's question.
	We remark that bounds of $n-1\le \boxicity(\Gamma(R))\le \thdim(\Gamma(R))\le n$ have been independently obtained in \cite{2026_Chandran_Compressed}.
	\begin{restatable}{theorem}{thmBoxThdReduced}\label{thm:box_thd_reduced}
	Let $R$ be a finite commutative reduced ring with non-zero identity, and let $n$ denote the number of minimal prime ideals of $R$.
	If $n=1$, then $\boxicity(\Gamma(R))=\thdim(\Gamma(R))=0$.
	For $n=2$, we have
	\[\boxicity(\Gamma(R))=\thdim(\Gamma(R))=\begin{cases}
		0 & \text{both minimal prime ideals are lonely}\\
		1 & \text{exactly one minimal prime ideal is lonely}\\
		2 & \text{neither is lonely}
	\end{cases}.
	\]
	For $n\ge 3$, we have
	\[\boxicity(\Gamma(R))=\thdim(\Gamma(R))=\begin{cases}
		n-1 & \text{there exists } i\in[n] \text{ s.t. $M_i$ is lonely}\\
		n & \text{otherwise}
	\end{cases}.
	\]
\end{restatable}
	Given the family $\{M_1, \ldots, M_n\}$ of minimal prime ideals of $R$, it is known that for every $j \in [n]$, $\bigcap_{i\in[n], \ i \neq j} M_i \setminus M_j$ is non-empty (see \cref{lemma:mu_surjective}).
	We call the minimal prime ideal $M_j$ \emph{lonely} when this set is a singleton, and non-lonely otherwise.
	
	For intuition, consider the reduced ring $R = \mathbb{Z}/30\mathbb{Z}$, whose minimal prime ideals are generated by
	$2$, $3$, and $5$.
	The prime ideal $2\mathbb{Z}/30\mathbb{Z}$ is lonely: $15$ is the unique element of
	$(3\mathbb{Z}/30\mathbb{Z} \cap 5\mathbb{Z}/30\mathbb{Z}) \setminus 2\mathbb{Z}/30\mathbb{Z}$.
	The prime ideal $3\mathbb{Z}/30\mathbb{Z}$, on the other hand, is not lonely: $(2\mathbb{Z}/30\mathbb{Z} \cap
	5\mathbb{Z}/30\mathbb{Z}) \setminus 3\mathbb{Z}/30\mathbb{Z} = \{10, 20\}$ contains two elements.
	Crucially, $10$ and $20$ are non-adjacent in $\Gamma(R)$, yet have the same set of zero divisors (i.e.\ the same neighborhood).
	Such pairs of `twin' non-adjacent vertices were used by Kavaskar~\cite{2025_Kavaskar} in an attempted proof of a lower bound on the boxicity of
	zero divisor graphs of reduced rings.
	As pointed out by Chandran and Sahoo~\cite{2025_Chandran_ZeroDivisor}, this proof was incorrect.
	
	Besides the notion of loneliness, our proof of \cref{thm:box_thd_reduced} develops on a connection
	between reduced zero divisor graphs and some combinatorial graph classes.
	To this end, it is helpful to consider the \emph{reduced zero divisor graph} $\Gamma_E(R)$.
	\begin{definition}
		Let $G$ be a graph. For $x,y \in V(G)$, write $x \sim y$ if they have the same sets of neighbors in $G$.
		This is an equivalence relation on $V(G)$, and each equivalence class induces an independent set in $G$.
		The \emph{reduced graph} $G_E$ of $G$ is the graph whose vertex set is the set of equivalence classes of $\sim$, with two distinct classes $[x]$ and $[y]$ being adjacent in $G_E$ if and only if $x$ and $y$ are adjacent in $G$.
	\end{definition}
	We denote the reduced graph of $\Gamma(R)$ by $\Gamma_E(R)$.
	Anderson and LaGrange's work on zero divisor graphs shows that, for a finite commutative
	reduced ring $R$ with $n$ minimal prime ideals, $\Gamma_E(R) \cong \Gamma(\mathbb{Z}_2^n)$
	\cite[Theorem~1.1]{ANDERSON20121626}.
	Under the natural identification of $\mathbb{Z}_2^n$ with $2^{[n]}$, two elements of
	$\mathbb{Z}_2^n \setminus \{\mathbf{0},\mathbf{1}\}$ multiply to $\mathbf{0}$ exactly when
	the corresponding subsets of $[n]$ are disjoint.
Hence, $\Gamma(\mathbb{Z}_2^n)$ is precisely the \emph{disjointness graph} $D_n$ on $2^{[n]}\setminus\{\emptyset,[n]\}$, i.e., the graph whose
	vertex set comprises the non-empty proper subsets of $[n]\coloneqq \{1,\dots,n\}$, and in which two subsets are adjacent if and only if they are disjoint.
	As part of our proof, we establish the boxicity of the disjointness graph (see \cref{subsec:proof_box_disjoint}).
	\begin{theorem}\label{theorem:boxicity_disjointness_graph}
	For $n\in \{1,2\}$, we have $\boxicity(D_n) = \thdim(D_n) = 0$. For $n\ge 3$, it holds that $\boxicity(D_n) = \thdim(D_n) = n-1$. 
	\end{theorem}
	%
	%
	\subsection{Quotients of principal ideal domains}
	Our second result addresses the boxicity of another class of zero divisor graphs.
	Improving on the bounds by Kavaskar~\cite{2025_Kavaskar}, Chandran and Sahoo~\cite{2025_Chandran_ZeroDivisor} established the boxicity and threshold dimension of the zero divisor graph of $\mathbb{Z}/M\mathbb{Z}$ for any $M\geq 2$.
	Kavaskar~\cite{2025_Kavaskar} showed that when $M = p^k$ for a prime $p$ and $k \ge 3$, we
	have $\boxicity(\Gamma(\mathbb{Z}/M\mathbb{Z})) = 1$, and when $k \le 2$,
	$\boxicity(\Gamma(\mathbb{Z}/M\mathbb{Z})) = 0$.
	It is easy to check that $\thdim(\Gamma(\mathbb{Z}/M\mathbb{Z}))$ agrees with $\boxicity(\Gamma(\mathbb{Z}/M\mathbb{Z}))$
	in both cases.
	Therefore, we only need to consider the case when $M$ has at least two prime divisors, for
	which Chandran and Sahoo obtained the following result.
	\begin{theorem}\label{theorem:box_ZnZ}
		Let $M = \prod_{i \in [n]} p_i^{m_i}$ where $m_i \geq 1$, $n \geq 2$, be the prime factorization of $M$.
		Then
		\[ n-1 = \boxicity(\Gamma(\mathbb{Z}/M\mathbb{Z})) \leq \thdim(\Gamma(\mathbb{Z}/M\mathbb{Z})) \leq n, \ \text{if} \ M\equiv2 \pmod{4} \ \text{and} \ m_i \leq 2 \ \text{for} \ i \in [n],\]
		and 
		\[ \boxicity(\Gamma(\mathbb{Z}/M\mathbb{Z})) = \thdim(\Gamma(\mathbb{Z}/M\mathbb{Z})) = n, \ \text{otherwise}.\]
	\end{theorem}

	As a common generalization of the rings $\mathbb{Z}/M\mathbb{Z}$ studied above, we also consider zero divisor graphs of quotient rings of principal ideal domains (PIDs).
	Let $R$ be a PID and let $I = gR \subsetneq R$ be a proper, non-prime ideal such that the quotient ring $Q \coloneqq R/I$ is
	finite. We characterize the boxicity and threshold dimension of $Q$ based on the prime factorization of $g$.
	When $R = \mathbb{Z}$, we recover (and extend) prior results for $\mathbb{Z}/M\mathbb{Z}$.
	Write the factorization of $g$ as
	$g = u \cdot \prod_{i=1}^n p_i^{m_i}$, where $u$ is a unit, $p_1,\ldots,p_n$ are pairwise
	non-equivalent primes, and $m_1,\ldots,m_n \ge 1$; since $I$ is proper and not prime, we have
	$n \ge 1$ and $\sum_{i=1}^n m_i \ge 2$.

	We determine $\boxicity(\Gamma(Q))$ and $\thdim(\Gamma(Q))$ in terms of $n$, the exponents $m_1,\ldots,m_n$, and the number of indices $i$ with $|R/p_iR|=2$ (see \cref{lemma:n_eq_1,theorem:pid_boxicity,theorem:pid_th}).
	
	As for reduced rings, our intuition comes from a combinatorial study of the graph class.
	We remark that, regardless of the primes that appear in the factorization of $g$, the main focus is on the exponents.
	Two elements $a,b$ have product zero if and only if, for each index $i \in [n]$, their exponent sums are at least $m_i$.
	This inspired the definition of the \emph{integral covering graph} $D(\bm{m})$, for $\bm{m}\in\mathbb{N}_{\geq 1}^n$: its vertex set is $\prod_{i=1}^n [m_i]_0 \setminus \{\bm{0},\bm{m}\}$, with two distinct vertices adjacent if and only if their coordinates sum to at least $m_i$ in every coordinate $i$ (\cref{def:integral_covering_graph}).
	We determine the boxicity of $D(\bm{m})$ for any $\bm{m} \in \mathbb{N}_{\geq 1}^n$ (\cref{theorem:boxicity_integral_covering,theorem:thdim_integral_covering,lemma:integral_covering_special_cases}).

	Exploiting a connection between the boxicities of $\Gamma_E(Q)$ and $D(\bm{m})$ (see \cref{cor:reduced_zero_divisor_graph_D_m,cor:reduced_zero_divisor_graph_D_m_n_2,cor:reduced_zero_divisor_graph_D_m_1_1,thm:boxicity_two_equiv_relations}) and handling the case $n=1$ separately (\cref{lemma:n_eq_1_reduced}) answers the following question of Chandran and Sahoo, even for general finite quotients of pids.
	\begin{question}[{\cite[Question~5.1]{2025_Chandran_ZeroDivisor}}]
		What is the exact value of $\boxicity(\Gamma_E(\mathbb{Z}_N))$, and when does
		$\boxicity(\Gamma_E(\mathbb{Z}_N)) = \boxicity(\Gamma(\mathbb{Z}_N))$?
	\end{question}
	While we were completing this work, Chandran and Sahoo~\cite{2026_Chandran_Compressed} independently answered the same question for $\mathbb{Z}_N$.

	
	\section{Preliminaries}
	
	\subsection{Graphs}
	All graphs in this paper are finite and simple.
	We use standard graph-theoretic notation throughout and briefly recall here the notation used most frequently.
	
	Let $G = (V, E)$ be a graph.
	For a vertex $v \in V$, the sets $\delta_G(v) \subseteq E$ and $N_G(v) \subseteq V$ are defined to be the edges incident to $v$ and vertices adjacent to $v$, respectively.
	The \emph{closed neighbourhood} of $v$ is defined as $N_G[v] = N_G(v) \cup \{v\}$.
	The \emph{degree} of a vertex $v$ is denoted by $d_G(v)$, and satisfies $d_G(v) = |\delta_G(v)| = |N_G(v)|$.
	A vertex $v$ is \emph{isolated} in $G$ if $\delta_G(v) = \emptyset$.
	A \emph{subgraph} of $G$ is a graph $(U,F)$ satisfying $U \subseteq V$ and $F \subseteq E$.
	A subgraph $(U,F)$ is \emph{induced} in $G$ if for all $u,v \in U$, $uv \in F$ if and only if $uv \in E$.
	The induced subgraph of $G$ on a vertex set $U \subseteq V$ is denoted by $G[U]$.
	
	\subsection{Rings, zero divisors and minimal ideals}
	In this paper, unless stated otherwise, a \emph{ring} will always be a commutative, finite, non-zero ring with $1$. While we recap some basic definitions and statements that we need in order to formulate our results, for a broader introduction to commutative algebra, we refer the reader to standard textbooks such as \cite{bourbaki1972commutative}.
	\begin{definition}
		Let $R$ be a ring. We call $r\in R$ a \emph{zero divisor} if there exists $s\in R\setminus \{0\}$ such that $r\cdot s=0$. We denote the set of zero divisors of $R$ by $Z(R)$.
	\end{definition}
	The object we are interested in this paper is the \emph{zero divisor graph} of a ring.
	\begin{definition}
		Let $R$ be a ring. Its \emph{zero divisor graph} $\Gamma(R)$ is the simple graph on the vertex set $Z(R)\setminus\{0\}$ in which two elements $r,s\in Z(R)\setminus \{0\}$ are connected by an edge if and only if $r\cdot s=0$.
	\end{definition}
	In this section, we will be concerned with the zero divisor graphs of \emph{reduced rings}, which do not feature \emph{nilpotent} zero divisors (except for $0$).
	\begin{definition}
		Let $R$ be a ring. We call $r\in R$ \emph{nilpotent} if there exists $n\in\mathbb{N}_{>0}$ such that $r^n=0$. We denote by $N(R)\coloneqq \{r\in R\colon r \text{ is nilpotent}\}$ the \emph{nilradical} of $R$.
		We call $R$ \emph{reduced} if $N(R)=\{0\}$.
	\end{definition}
	
	It is well-known that in reduced rings, the set of zero divisors can be characterized as the union of all minimal prime ideals (see \cref{cor:union_of_minimal_ideals_reduced}).
	To formalize this relation, we require the following algebraic preliminaries.
	\begin{definition}
		Let $R$ be a ring.
		We call $I\subseteq R$ an \emph{ideal} in $R$ if $(I,+)$ forms a subgroup of $(R,+)$ such that $r\cdot i\in I$ for every $r\in R$ and $i\in I$. If, in addition, for every $r,s\in R$ with $r\cdot s\in I$, we have $r\in I$ or $s\in I$, we call $I$ a \emph{prime ideal}. We denote by $\Min(R)$ the collection of $\subseteq$-minimal prime ideals of $R$.
	\end{definition}
	We state the following property of prime ideals.
	\begin{lemma}[follows from \cite{bourbaki1972commutative} Chapter 2 \S 1 Proposition 1]\label{lemma:prime_ideal_containment}
		Let $R$ be a ring, let $P$ be a prime ideal of $R$, let $k\in\mathbb{N}$ and $(I_j)_{j\in [k]}$ be ideals of $R$ such that $\bigcap_{j\in [k]} I_j\subseteq P$. Then there is $j\in[k]$ with $I_j\subseteq P$.
	\end{lemma}
	One can show that all elements of minimal prime ideals are zero divisors.
	
	\begin{lemma}[follows from \cite{bourbaki1972commutative} Chapter 2 \S 2 Proposition 12]\label{lemma:minimal_prime_ideals_consist_of_zero_divisors}
		Let $R$ be a ring and let $M\subseteq R$ be a minimal prime ideal. Then $M\subseteq Z(R)$.
	\end{lemma}
	The elements that appear in all minimal prime ideals are precisely the nilpotent elements. In particular, in a reduced ring, $0$ is the only ring element appearing in all minimal prime ideals.

	\begin{theorem}[\cite{bourbaki1972commutative} Chapter 2 \S 2 Proposition 13]\label{theorem:intersection_of_minimum_ideals}
		Let $R$ be a ring. Then $\bigcap \Min(R)=N(R)$.
	\end{theorem}
	
	\begin{corollary}\label{cor:reduced_intersection_min_prime_ideals}
		If $R$ is a reduced ring, then $\bigcap \Min(R)=\{0\}$.
	\end{corollary}
	
	\cref{lemma:minimal_prime_ideals_consist_of_zero_divisors,cor:reduced_intersection_min_prime_ideals} imply the following characterization of the set of zero divisors of a reduced ring.
	
	\begin{corollary}\label{cor:union_of_minimal_ideals_reduced}
		Let $R$ be a reduced ring. Then $\bigcup \Min(R)=Z(R)$.
	\end{corollary}
	\begin{proof}
		By \cref{lemma:minimal_prime_ideals_consist_of_zero_divisors}, we have $\bigcup \Min(R)\subseteq Z(R)$. Next, let $r\in Z(R)$. Then there is $s\in R\setminus\{0\}$ such that $r\cdot s=0$. By \cref{cor:reduced_intersection_min_prime_ideals}, there exists $M\in\Min(R)$ such that $s\notin M$. As $M$ is an ideal, we have $r\cdot s=0\in M$, and the fact that $M$ is prime yields $r\in M$.
	\end{proof}
	
	\subsection{Boxicity}
	Boxicity is monotone under taking induced subgraphs.
	\begin{lemma}\label{lemma:induced_subgraph}
		Let $G$ be a graph and let $H$ be an induced subgraph of $G$.
		Then
		\[
		\boxicity(H) \leq \boxicity(G).
		\]
	\end{lemma}
	We further recall that a graph has boxicity at most $d$ if and only if it can be written as the intersection of $d$ interval graphs $(I_j)_{j\in [d]}$ (see \cref{lemma:boxicity_interval_graphs}).
	An immediate consequence is the following observation.
	For every non-edge $f \in \binom{V(G)}{2} \setminus E(G)$, there exists an index $j \in [d]$ such that the interval graph $I_j$ does not contain $f$.
	In this case, we say that the interval graph $I_j$ \emph{deletes} the non-edge $f$.

	The \emph{join} of $n$ vertex-disjoint graphs $\Gamma_1,\Gamma_2,\dots,\Gamma_n$, denoted by $\bigvee_{i=1}^n \Gamma_i$, is the graph with vertex set
	\[
	V\Big(\bigvee_{i=1}^n \Gamma_i\Big) = \bigcup_{i=1}^n V(\Gamma_i)
	\]
	and edge set
	\[
	E\Big(\bigvee_{i=1}^n \Gamma_i\Big) = \bigcup_{i=1}^n E(\Gamma_i) \;\cup\; 
	\big\{\{v_i,v_j\} : v_i\in V(\Gamma_i),\ v_j\in V(\Gamma_j),\ 1\le i<j\le n\big\}.
	\]
	The next lemma shows that boxicity is additive under the join operation.
	
	\begin{lemma}[\cite{1983_Cozzens}]\label{lemma:disjoint_operation}
		Let $\Gamma_1,\Gamma_2,\dots,\Gamma_n$ be $n$ vertex-disjoint graphs.
		Then
		\[
		\boxicity\Big(\bigvee_{i=1}^n \Gamma_i\Big) = \sum_{i=1}^n \boxicity(\Gamma_i).
		\]
	\end{lemma}

	We further state the following useful characterization of threshold graphs in terms of forbidden subgraphs.
	\begin{theorem}[\cite{1977_Chvatal}]\label{theorem:forbidden_subgraphs_threshold}
		Let $G$ be a graph.
		Then $G$ is a threshold graph if and only if $G$ has no induced subgraphs isomorphic to $C_4$, $P_4$, or $2K_2$ (see \cref{fig:forbidden_threshold}). 
	\end{theorem}
	\begin{figure}
		\centering
		\begin{tikzpicture}[
			vertex/.style={circle, draw, fill=black, inner sep=0pt, minimum size=5pt},
			every edge/.style={draw, thick},
			label distance=4pt
			]
			
			\begin{scope}[xshift=0cm]
				\node[vertex] (a1) at (0,  1) {};
				\node[vertex] (a2) at (1,  1) {};
				\node[vertex] (a3) at (1,  0) {};
				\node[vertex] (a4) at (0,  0) {};
				
				\draw[thick] (a1)--(a2)--(a3)--(a4)--(a1);
				
				\node at (0.5, -0.6) {$C_4$};
			\end{scope}
			
			\begin{scope}[xshift=2.5cm]
				\node[vertex] (b1) at (0,   0.5) {};
				\node[vertex] (b2) at (0.7, 0.5) {};
				\node[vertex] (b3) at (1.4, 0.5) {};
				\node[vertex] (b4) at (2.1, 0.5) {};
				
				\draw[thick] (b1)--(b2)--(b3)--(b4);
				
				\node at (1.05, -0.6) {$P_4$};
			\end{scope}
			
			\begin{scope}[xshift=5.8cm]
				\node[vertex] (c1) at (0,   1) {};
				\node[vertex] (c2) at (0,   0) {};
				\node[vertex] (c3) at (0.9, 1) {};
				\node[vertex] (c4) at (0.9, 0) {};
				
				\draw[thick] (c1)--(c2);
				\draw[thick] (c3)--(c4);
				
				\node at (0.45, -0.6) {$2K_2$};
			\end{scope}
			
		\end{tikzpicture}
		\caption{Forbidden induced subgraphs in a threshold graph.}
		\label{fig:forbidden_threshold}
	\end{figure}
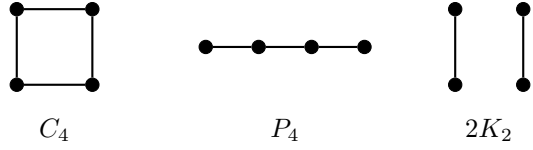
	We conclude this section by introducing the notion of a \emph{partial join}, which helps us to bound the threshold dimension of the graphs we encounter.
	\begin{definition}\label{def:partial_join}
		Let $n\in\mathbb{N}$ and let $(G_i=(V_i,E_i))_{i\in [n]}$ be graphs. We call a graph $G=(V,E)$ a \emph{partial join} of $(G_i)_{i\in [n]}$ if there exist a bijection $\sigma\colon \dot{\bigcup}_{i\in [n]} V_i\leftrightarrow V$ with the following properties:
		\begin{enumerate}
			\item For every $i\in [n]$, $\sigma$ induces a graph isomorphism between $G_i$ and $G[\sigma(V_i)]$.
			\item For every $1\le i < j\le n$, the set of non-edges $M_{i,j}\coloneqq \{\{v,w\} \notin E\colon v\in \sigma(V_i), w\in \sigma(V_j)\}$ is a matching.
		\end{enumerate}
	\end{definition}
	\begin{remark}
		In fact, it suffices to require that $G_i$ is isomorphic to $G[\sigma(V_i)]$, without the isomorphism being induced by $\sigma$.
	\end{remark}
	
	\begin{theorem}\label{theorem:partial_join}
		Let $n\in\mathbb{N}$, let $(G_i=(V_i,E_i))_{i\in [n]}$ be graphs, and let $G=(V,E)$ be a partial join of $(G_i)_{i\in [n]}$. Then $\thdim(G)\ge \sum_{i\in [n]} \thdim(G_i)$.
	\end{theorem}
	\begin{proof}
		Let $\thdim(G)\coloneqq d$ and let $(H_i=(V,F_i))_{i\in [d]}$ be threshold graphs with $\bigcap_{i\in [d]} F_i = E$. For $j\in [n]$, let $I_j\coloneqq \{i\in [d]\colon \exists v\ne w\in \sigma(V_j)\colon \{v,w\}\notin F_i\}$ be the set of indices $i$ such that $H_i$ contains a `non-edge' of $G[\sigma(V_j)]\cong G_j$. Then
		$E[\sigma(V_j)]=\bigcap_{i\in I_j} F_i[\sigma(V_j)]$, so $\thdim(G_j)=\thdim(G[\sigma(V_j)])\le |I_j|$. Hence, it suffices to show that the sets $(I_j)_{j\in [n]}$ are pairwise disjoint to prove the desired statement.
		
		Assume towards a contradiction that there were $j\ne k \in [n]$ with $I_j\cap I_k\ne \emptyset$ and let $i\in I_j\cap I_k$.
		Let $v\ne w\in \sigma(V_j)$ and $x\ne y\in \sigma(V_k)$ such that $\{v,w\}\notin F_i$ and $\{x,y\}\notin F_i$.
		As $E \subseteq F_i$, every non-edge in $H_i$ is a non-edge in $G$.
		Thus, by definition of a partial join, $M\coloneqq \{\{a,b\}\colon a\in \{v,w\},b\in\{x,y\},\{a,b\}\notin F_i\}$ forms a matching.
		It is easy to observe that: if $|M|=0$, then $H_i[\{v,w,x,y\}]\cong C_4$; if $|M|=1$, then $H_i[\{v,w,x,y\}]\cong P_4$; and if $|M|=2$, then $H_i[\{v,w,x,y\}]\cong 2K_2$.
		By \cref{theorem:forbidden_subgraphs_threshold}, the graph $H_i[\{v,w,x,y\}]$ is not a threshold graph in any of the three cases.
		As threshold graphs are closed under taking induced subgraphs, this contradicts the assumption that $H_i$ is a threshold graph.
	\end{proof}
	
	\subsection{Reduced graphs}
	In this section we study the relationship between the boxicity of a graph and that of its reduced graph, which is obtained by identifying vertices with the same neighborhood structure.

	\begin{definition}
		Given a graph $G=(V,E)$, we define an equivalence relation $\sim$ on $V$ by setting $u\sim v$ if $N(u)=N(v)$.
		For $v\in V$, denote by $[v]$ the equivalence class of $v$ with respect to $\sim$. 
		The reduced graph $G_E$ has as vertex set $\{[v]\colon v\in V\}$ the equivalence classes with respect to $\sim$, and edge set $\{\{[u],[v]\}\colon \{u,v\}\in E\}$.
	\end{definition}
In order to study the boxicity of reduced graphs, the following relation will be useful.
	\begin{definition}
		Given a graph $G=(V,E)$, we define the relation $\sim'$ on $V$ by letting $u\sim' v$ if $N(u)\setminus \{u,v\}=N(v)\setminus\{u,v\}$.
	\end{definition}
	Once established that $\sim'$ is also an equivalence relation, we can define a graph analogous to the reduced graph $G_E$.
	\begin{proposition}
		$\sim'$ is an equivalence relation.
	\end{proposition}
	\begin{proof}
		Clearly $v\sim' v$ for every $v\in V$ and $\sim'$ is symmetric. Let $u,v,w\in V$ be three distinct vertices with $u\sim ' v$ and $v\sim' w$. Then $N(u)\setminus \{u,v,w\}=N(v)\setminus \{u,v,w\}=N(w)\setminus \{u,v,w\}$. Hence, it remains to show that $v\in N(u)$ if and only if $v\in N(w)$. By symmetry, it suffices to show that $v\in N(u)$ implies $v\in N(v)$. Suppose $v\in N(u)$. Then $u\in N(v)\setminus \{v,w\}=N(w)\setminus \{v,w\}$, so $u\in N(w)$, i.e., $w\in N(u)\setminus \{u,v\}=N(v)\setminus \{u,v\}$. Hence, $v\in N(w)$.
	\end{proof}
	
	\begin{definition}
		For $v\in V$, we let $[v]'$ denote the equivalence class of $v$ with respect to $\sim'$.
		The graph $G_{E'}$ is given by $V(G_{E'})=\{[v]'\colon v\in V\}$ and $E(G_{E'})=\{\{[u]',[v]'\}\colon [u]'\ne [v]' \text{ and } \{u,v\}\in E\}$.
	\end{definition}
	
	Note that for a graph $G$, its reduced graph $G_E$ is isomorphic to an induced subgraph of $G$.
	Hence, $\boxicity(G) \geq \boxicity(G_E)$ by \cref{lemma:induced_subgraph}.
	In particular, the cycle on four vertices has boxicity $2$, while its reduced graph, which is a clique on two vertices, has boxicity $0$.
	This shows that the above inequality can be strict.
	We now show that going from $G_E$ to $G_{E'}$ the boxicity does not change.
	To achieve this, we give a closer look at the equivalence classes in $V/\sim'$.
	\begin{proposition}\label{prop:clique_or_stable}
		Every equivalence class of $\sim'$ induces a clique or a stable set of $G$.
	\end{proposition}
	\begin{proof}
		This is clear for equivalence classes of size $1$ and $2$.
		Consider an equivalence class $U\subseteq V$ of size $|U|\ge 3$.
		If $U$ is a stable set in $G$, we are done.
		Hence, let $v,w\in U$ with $\{v,w\}\in E$.
		Let $x,y\in U$ be distinct vertices.
		We need to show that $\{x,y\}\in E$.
		This is clear if $\{x,y\}=\{v,w\}$.
		If $|\{x,y\}\cap \{v,w\}|=1$, let w.l.o.g.\ $v=x$ and $w\ne y$.
		As $w\sim' y$ and $v\in N(w)\setminus \{w,y\}=N(y)\setminus \{w,y\}$, we have $\{x,y\}=\{v,y\}\in E$.
		Finally, assume $\{v,w\}\cap \{x,y\}=\emptyset$.
		By the previously considered case, we know that $\{v,x\}\in E$.
		Now, again applying this case yields $\{x,y\}\in E$.
	\end{proof}
	
	\begin{proposition}\label{prop:stable_equivalence_class}
		Every equivalence class of $\sim'$ that is a stable set of $G$ is an equivalence class of $\sim$.
	\end{proposition}
	\begin{proof}
		As $u\sim v$ implies $u\sim' v$, the equivalence classes of $\sim$ are subsets of equivalence classes of $\sim'$. On the other hand, for two vertices $u$ and $v$ contained in an equivalence class of $\sim'$ that is a stable set, we have $N(u)=N(u)\setminus \{u,v\}=N(v)\setminus \{u,v\}=N(v)$, so $u\sim v$.
	\end{proof}
	
	\begin{proposition}\label{prop:clique_equivalence_class}
		For every equivalence class $U$ of $\sim'$ that forms a clique in $G$ and every $u\in U$, we have $[u]=\{u\}$.
	\end{proposition}
	\begin{proof}
		As $u\sim v$ implies $u\sim' v$, we have $[u]\subseteq U$. For every $v\in U\setminus \{u\}$, we have $u\in N(v)\setminus N(u)$, so $u\not\sim v$. Hence, $[u]=\{u\}$.
	\end{proof}

	\begin{theorem}\label{thm:boxicity_two_equiv_relations}
		We have $\boxicity(G_E)=\boxicity(G_{E'})$.
	\end{theorem}
	\begin{proof}
		By \cref{prop:stable_equivalence_class,prop:clique_equivalence_class}, $G_{E'}$ arises from $G_{E}$ by, for every equivalence class $U$ of $\sim'$ that forms a clique in $G$, contracting $\{[u]\colon u\in U\}$ into a single vertex.
		Note that $\{[u]\colon u\in U\}$ forms a clique in $G_{E}$ and that for every $u,u'\in U$ and $v\in V\setminus U$, $\{[u],[v]\}\in E(G_{E})$ if and only if $\{[u'],[v]\}\in E(G_{E'})$.
		In particular, $G_{E'}$ is isomorphic to the induced subgraph of $G_{E}$ we obtain by choosing exactly one of the vertices $[u]$ with $u\in U$ for each equivalence class $U$ of $\sim'$ that forms a clique.
		This implies $\boxicity(G_E)\ge\boxicity(G_{E'})$ by \cref{lemma:induced_subgraph}.
		For the other inequality, let $\boxicity(G_{E'})\eqqcolon d$ and let $H'_1,\dots,H'_d$ be interval graphs whose intersection is $G_{E'}$.
		We obtain interval graphs $H_1,\dots,H_d$, by, for replacing an interval $I_U$ in the interval representation of $H_i$ corresponding to an equivalence class $U$ of $\sim'$ that forms a clique by $|U|$ copies $I_{[u]}$ for $u\in U$.
		The graphs $H_1,\dots,H_d$ yield an interval representation of $G_{E}$.
	\end{proof}

	\section{Integral covering graphs}
	In this section, we introduce the notion of \emph{integral covering graphs}. While being defined in a completely combinatorial way, integral covering graphs capture the structure that lies at the heart of the zero divisor graphs we study in this paper: We will observe that the reduced graphs of the zero divisor graphs we investigate are isomorphic to integral covering graphs. 
	
	In the following, we will use bold letters, e.g.\ $\bm{m}$, to denote vectors, and we will write $m_i$ to denote the $i$-th coordinate of $\bm{m}$. We further use $\bm{0}_n$, $\bm{1}_n$ and $\bm{2}_n$ to refer to the $n$-dimensional all-zero, all-one  and all-two vector, respectively. If the dimension $n$ is clear from the context, we may simply write, $\bm{0}$, $\bm{1}$ and $\bm{2}$, respectively.
	\begin{definition}\label{def:integral_covering_graph}
		Given $\bm{m}\in\mathbb{N}_{\ge 1}^n$, the \emph{integral covering graph} $D(\bm{m})$ is defined as follows: The vertex set is given by $V(D(\bm{m}))=\prod_{i=1}^n [m_i]_0\setminus \{\bm{0},\bm{m}\}$.
		Two distinct vertices $\bm{v}$ and $\bm{w}$ are connected by an edge if $v_i+w_i\ge m_i$ for all $i\in [n]$.
	\end{definition}
	We remark that for $\bm{m}\in\mathbb{N}_{\ge 1}^n$, the induced subgraph $D(\bm{m})\left[\prod_{i=1}^n \{0,m_i\}\setminus \{\bm{0},\bm{m}\}\right]$ is isomorphic to $D(\bm{1}_n)$.
	This implies \cref{prop:lower_bound_via_D_n}.
	\begin{proposition}\label{prop:lower_bound_via_D_n}
		For every $n\in\mathbb{N}_{\ge 1}$ and every $\bm{m}\in\mathbb{N}_{\ge 1}^n$, we have $\boxicity(D(\bm{m}))\ge \boxicity(D(\bm{1}_n))$ and $\thdim(D(\bm{m}))\ge \thdim(D(\bm{1}_n))$.
	\end{proposition}
	On the other hand, the definition of $D(\bm{m})$ implies that it arises as the intersection of $n$ threshold graphs. 
	\begin{proposition}\label{prop:generalized_disjointness_graph_upper_bound}
		For every $n\in\mathbb{N}_{\ge 1}$ and every $\bm{m}\in\mathbb{N}_{\ge 1}^n$, we have $\boxicity(D(\bm{m}))\le \thdim(D(\bm{m}))\le n$.
	\end{proposition}
	We remark that the notion of integral covering graphs generalizes the notion of the disjointness graph $D_n$ on the vertex set $2^{[n]}\setminus \{\emptyset, [n]\}$, in which two sets are adjacent if and only if they are disjoint. In particular, \cref{theorem:boxicity_disjointness_graph} follows from the results in this section.
	\begin{proposition}\label{prop:D_n_iso}
	The disjointness graph $D_n$ is isomorphic to $D(\bm{1}_n)$.
	\end{proposition}
	\begin{proof}
	We identify the vertex set of $D(\bm{1}_n)$ with $2^{[n]}\setminus \{\emptyset, [n]\}$, the collection of non-empty proper subsets of $[n]$, by identifying $s\subseteq [n]$ with the vector $\bm{s}\in \{0,1\}^n$ with $\bm{s}_i=0$ if and only if $i\in s$. This yields an isomorphism between $D(\bm{1}_n)$ and $D_n$.
	\end{proof}
		
	In the following, we will determine the boxicity and threshold dimension of integral covering graphs. 
	We begin by handling two corner cases, namely the case $n=1$, as well as the case $n=2$ and $m_1=m_2=1$.
	\begin{lemma}\label{lemma:integral_covering_special_cases}
		We have $\boxicity(D(1,1))=\thdim(D(1,1))=0$. For $m\in\mathbb{N}_{\ge 1}$, we have 
		\[\boxicity(D(m))=\thdim(D(m))=\begin{cases}
			0 & m\le 3\\
			1 & m\ge 4
		\end{cases}.\]
		
	\end{lemma}
	\begin{proof} For the first statement, note that $D(1,1)$ consists of two adjacent vertices, namely $(1,0)$ and $(0,1)$.
		
		For the second statement, we recall that $D(m)$ is a threshold graph. So its boxicity and threshold dimension are equal to $0$, if $D(m)$ constitutes a clique, and $1$ otherwise. The number of vertices in $D(1)$ and $D(2)$ equals $0$ and $1$, respectively, and the graph $D(3)$ contains precisely the two vertices $1$ and $2$, which are adjacent. For $m\ge 4$, $1$ and $2$ are non-adjacent vertices of $D(m)$.	
	\end{proof}
	
	\cref{lemma:lower_bound_n_minus_1} tells us that in all remaining cases, the boxicity and threshold dimension of $D(\bm{m})$ are at least $n-1$. We prove it in \cref{sec:lower_bound_n_minus_1}.
	\begin{lemma}\label{lemma:lower_bound_n_minus_1}
		Let $n\ge 2$ and let $\bm{m}\in\mathbb{N}_{\ge 1}^n$ with $\bm{m}\ne \bm{1}_2$. Then $\thdim(D(\bm{m}))\ge \boxicity(D(\bm{m}))\ge n-1$.
	\end{lemma}
	Hence, it remains to distinguish the cases in which the boxicity/threshold dimension are $n$, and those in which they are $n-1$. Together with \cref{lemma:integral_covering_special_cases}, \cref{theorem:boxicity_integral_covering,theorem:thdim_integral_covering} provide a full characterization of the boxicity and threshold dimension of integral covering graphs.
	\begin{theorem}\label{theorem:boxicity_integral_covering}
		Let $n\ge 2$ and let $\bm{m}\in \mathbb{N}_{\ge 1}^n$. Then \[\boxicity(D(\bm{m}))=\begin{cases}
			n & \exists i\in [n]\colon m_i \ge  4\\
			n & \exists i\in [n] \colon m_i=3 \text{ and } \exists j\in [n]\setminus \{i\}\colon m_j\ge 2\\
			n-1 & \exists i\in [n]\colon m_i=3 \text{ and } \forall j\in [n]\setminus \{i\}\colon m_j=1\\
			n & \bm{m}= \bm{2}_n\\
			n-1 & \bm{m}\in\{1,2\}^n \text{ and } \bm{m}\ne \bm{2}_n \text{ and } \bm{m}\ne \bm{1}_2\\
			0 & \bm{m}=\bm{1}_2
		\end{cases}.\]
	\end{theorem}
	\begin{theorem}\label{theorem:thdim_integral_covering}
		Let $n\ge 2$ and let $\bm{m}\in \mathbb{N}_{\ge 1}^n$. Then \[\thdim(D(\bm{m}))=\begin{cases}
			n & \bm{m}\ne \bm{1}_n\\
			0 & \bm{m}= \bm{1}_n \text{ and } n=2\\
			n-1 & \bm{m}= \bm{1}_n \text{ and } n\ge 3
		\end{cases}.\]
	\end{theorem}
	We develop the corresponding upper bounds in \cref{sec:integral_covering_upper_bounds}, and the necessary lower bounds in \cref{sec:intergral_covering_lower_bounds}. We then prove \cref{theorem:boxicity_integral_covering} in \cref{sec:proof_boxicity_integral_covering} and \cref{theorem:thdim_integral_covering} in \cref{sec:proof_thdim_integral_covering}.
	
	\subsection{Proof of \cref{lemma:lower_bound_n_minus_1}}\label{sec:lower_bound_n_minus_1}
	We first handle the case $n=2$. As $\bm{m}\ne \bm{1}_2$, we must have $m_1\ge 2$ or $m_2\ge 2$. Hence, $(1,0)$ and $(0,1)$ form two non-adjacent vertices of $D(\bm{m})$, implying that $\thdim(D(\bm{m}))\ge \boxicity(D(\bm{m}))\ge 1$.
	
	By \cref{prop:lower_bound_via_D_n}, it remains to show that $\boxicity(D(\bm{1}_n))\ge n-1$ for $n\ge 3$. The graph $D(\bm{1}_3)$, which is depicted in \cref{figure:D_3}, is known to have boxicity $2$.
	
	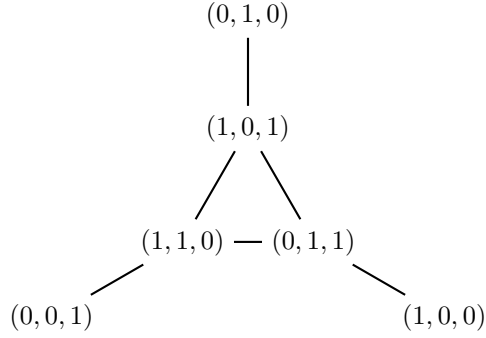
\begin{figure}
		\centering
		\begin{tikzpicture}
			\node (A) at ({cos(-30)},{sin(-30)}) {$(0,1,1)$};
			\node (B) at ({cos(90)},{sin(90)}) {$(1,0,1)$};
			\node (C) at ({cos(210)},{sin(210)}) {$(1,1,0)$};
			\draw[thick] (A)--(B)--(C)--(A);
			\node (A1) at ({3*cos(-30)},{3*sin(-30)}) {$(1,0,0)$};
			\node (B1) at ({2.5*cos(90)},{2.5*sin(90)}) {$(0,1,0)$};
			\node (C1) at ({3*cos(210)},{3*sin(210)}) {$(0,0,1)$};
			\draw[thick] (A)--(A1);
			\draw[thick] (B)--(B1);
			\draw[thick] (C)--(C1);
		\end{tikzpicture}
		\caption{The graph $D(\bm{1}_3)$.}\label{figure:D_3}
	\end{figure}
	
	So we may assume that $n\ge 4$ in the following.
	By \cref{prop:lower_bound_via_D_n,prop:D_n_iso}, it suffices to prove that $\boxicity(D_n)\ge n-1$, where $D_n$ denotes the disjointness graph on the vertex set $2^{[n]}\setminus\{\emptyset,[n]\}$.

	In order to prove a lower bound on the boxicity of $D_n$, we will show that each interval graph cannot delete `too many' of the non-edges in $D_n$. In doing so, we will, in fact only focus on the non-edges between pairs $\{a,b\}\subseteq [n]$ and the contained singleton sets $\{a\}$ and $\{b\}$. 
	More precisely, let $H$ be an interval graph on the vertex set $V(D_n)$ that contains $D_n$ as a subgraph.
	We say that $H$ \emph{handles} a pair $\{a,b\}\in \binom{[n]}{2}$ if neither $\{\{a\},\{a,b\}\}$ nor $\{\{b\},\{a,b\}\}$ is contained in $E(H)$.
	We say that $H$ \emph{touches} the pair $\{a,b\}$ if at most one of $\{\{a\},\{a,b\}\}$ nor $\{\{b\},\{a,b\}\}$ is contained in $E(H)$.
	In the remainder of this section, when we say `an interval graph', we mean an interval graph on the vertex set $V(D_n)$ that forms a supergraph of $D_n$.
	
	\begin{lemma}\label{lemma:disjoint_pairs}
		An interval graph cannot touch two disjoint pairs in $\binom{[n]}{2}$.
	\end{lemma}
	\begin{proof}
		Let $\{a,b\}$ and $\{c,d\}$ be two disjoint pairs in $[n]$.
		Let $H$ be an interval graph that touches $\{a,b\}$.
		Let $\mathcal{I}$ be an interval representation of $H$, and assume without loss of generality that the intervals $I_{\{a\}}$ and $I_{\{a,b\}}$ are disjoint.
		As $H$ is a supergraph of $D_n$ and $\{a,b\} \cap \{c,d\} = \emptyset$, we know that each of the intervals $I_{\{c\}}$, $I_{\{d\}}$, and $I_{\{c,d\}}$, intersect each of the intervals $I_{\{a\}}$ and $I_{\{a,b\}}$.
		In particular, they all cover the space between  $I_{\{a\}}$ and $I_{\{a,b\}}$, so they are pairwise intersecting.
		Hence, $H$ cannot touch $\{c,d\}$.
	\end{proof}
	
	\begin{lemma}\label{lemma:pairwise_intersecting_pairs}
		An interval graph can touch at most $n-1$ pairs in $\binom{[n]}{2}$.
	\end{lemma}
	\begin{proof}
		Let $\mathcal{F}\subseteq \binom{[n]}{2}$ be the set of pairs touched by an interval graph. 
		By Lemma~\ref{lemma:disjoint_pairs}, $\mathcal{F}$ consist of pairwise intersecting pairs.
		If $|\mathcal{F}|\le 1$, there is nothing to show, so suppose that $\mathcal{F}$ contains two intersecting pairs $\{a,b\}$ and $\{a,c\}$.
		If $\{b,c\}\in \mathcal{F}$, then $\mathcal{F}$ cannot contain any further pair, so $|\mathcal{F}|=3\le n-1$.
		If $\{b,c\}\notin\mathcal{F}$, then all pairs in $\mathcal{F}$ are of the form $\{a,d\}$ for some $d\in[n]\setminus \{a\}$ because the only pair not of this form intersecting both $\{a,b\}$ and $\{a,c\}$ is $\{b,c\}$.
		Again, $|\mathcal{F}|\le n-1$.
	\end{proof}
	
	By analyzing the ways intervals can be placed in order to handle pairs, we further prove the following result.
	\begin{lemma}\label{lemma:handle_at_most_2}
		An interval graph can handle at most $2$ pairs in $\binom{[n]}{2}$.
	\end{lemma}
	\begin{proof}
		Let $H$ be an interval graph and $\mathcal{I}$ be its interval representation.
		For each set $S$, we denote its interval by $I_S \eqqcolon [\ell_{S}, r_{S}]$.
		As $H$ is a supergraph of $D_n$, we have
		\begin{equation}\label{eq:disjoint_set_intersect}
			I_S \cap I_{S'} \neq \emptyset \text{ if } S \cap S' = \emptyset.
		\end{equation}
		
		We say that a pair $\{a,b\}$ is left-handled (respectively, right-handled) if $H$ handles $\{a,b\}$, and $I_{\{a,b\}}$ is entirely to the left (respectively, to the right) of $I_{\{a\}}$ and $I_{\{b\}}$
		We show that $\Iscr$ can left-handle at most one pair.
		An analogous argument applies to right-handled pairs, and since every handled pair is either left- or right-handled, this proves the lemma.
		
		Assume that for a  contradiction that $\Iscr$ left-handles two distinct pairs.
		By \cref{lemma:disjoint_pairs}, these two pairs must be of the form $\{a,b\}$ and $\{a,c\}$.
		As $\{a,b\}$ is left-handled, we have
		\begin{equation}\label{eq:int_a_int_b_right_int_ab}
			r_{\{a,b\}} < \min \{\ell_{\{a\}}, \ell_{\{b\}}\}.
		\end{equation}
		By \eqref{eq:disjoint_set_intersect}, the interval $I_{\{c\}}$ intersects each of the intervals $I_{\{a\}}$, $I_{\{b\}}$, and $I_{\{a,b\}}$.
		Together with \eqref{eq:int_a_int_b_right_int_ab}, this implies
		\begin{equation}\label{eq:int_c_in_between_pair_a_b}
			\ell_{\{c\}} \leq r_{\{a,b\}} < \min \{\ell_{\{a\}}, \ell_{\{b\}}\} \leq \ell_{\{b\}}.
		\end{equation}
		Now consider the second pair $\{a,c\}$.
		As it is also left-handled, we have
		\begin{equation}\label{eq:int_a_int_c_right_int_ac}
			r_{\{a,c\}} < \min \{\ell_{\{a\}}, \ell_{\{c\}}\}.
		\end{equation}
		Combining \eqref{eq:int_c_in_between_pair_a_b} and \eqref{eq:int_a_int_c_right_int_ac}, we obtain
		\begin{equation*}
			r_{\{a,c\}} < \ell_{\{c\}} \leq \ell_{\{b\}},
		\end{equation*}
		which shows that $I_{\{a,c\}}$ and $I_{\{b\}}$ are disjoint, contradicting \eqref{eq:disjoint_set_intersect}.
	\end{proof}
	
	We are now ready to complete the proof of \cref{lemma:lower_bound_n_minus_1} and show that the boxicity of $D_n$ is at least $n-1$ for $n\ge 4$.
	Let $d \coloneqq \boxicity(D_n)$, and let $H_1,\dots,H_d$ be interval graphs on the vertex set $V(D_n)$ that constitute supergraphs of $D_n$ and whose intersection is $D_n$.
	In total, there are $\binom{n}{2}\cdot 2=n\cdot(n-1)$ many non-edges of the form $\{\{a\},\{a,b\}\}$ that need to be deleted by the interval graphs $H_1,\dots,H_d$.
	For an interval graph $H_i$, the number of these non-edges deleted by $H_i$ is equal to the number of pairs $H_i$ touches plus the number of pairs that it handles because for every pair that is touched but not handled, exactly one non-edge is deleted, and for every pair that is handled (and touched), both of its non-edges are deleted.
	By \cref{lemma:disjoint_pairs,lemma:handle_at_most_2}, each $H_i$ deletes at most $n-1+2=n+1$ non-edges. This implies \[d\ge \frac{n\cdot(n-1)}{n+1}=\frac{(n+1)\cdot(n-2)+2}{n+1}=n-2+\frac{2}{n+1}>n-2.\]
	As $d$ is an integer, this implies $d\ge n-1$, so $D_n$ has boxicity at least $n-1$.
	\subsection{Upper bounds on boxicity and threshold dimension\label{sec:integral_covering_upper_bounds}}
	\begin{lemma}\label{lemma:box_th_1_2}
		Let $G=(V,E)$ be the intersection of $n\ge 2$ threshold graphs $T(\mu_i,m_i)_{i=1}^n$, where $m_i\in \{1,2\}$ for $i\in [n-1]$, $m_n=1$, and $\mu_i\colon V\rightarrow[m_i]_0$. Suppose further that there exists at most one $v\in V$ such that $\mu_i(v)=m_i$ for $i\in [n-1]$ and $\mu_n(v)=0$. Then we have $\boxicity(G)\le n-1$. If further $m_i=1$ for all $i\in [n]$, then also $\thdim(G)\le n-1$.
	\end{lemma}
	\begin{proof}
		We construct interval graphs $(H_i)_{i=1}^{n-1}$ with vertex set $V$ as follows:
		
		If $m_i=1$, then $H_i$ will, in fact, be a threshold graph; namely, we let $H_i\coloneqq T(\rho_i,1)$, where $\rho_i\colon V\rightarrow [0,1]$ is given by
		\[\rho_i(v)=\begin{cases}
			1 & \mu_i(v)=\mu_n(v)=1\\
			\frac{2}{3} & \mu_i(v)=1, \mu_n(v)=0\\
			\frac{1}{3} & \mu_i(v)=0, \mu_n(v)=1\\
			\frac{1}{6} & \mu_i(v)=\mu_n(v)=0
		\end{cases}.\]
		If $m_i=2$, then we define $H_i$ as follows:
		\begin{itemize}
			\item We assign all $v\in V$ with $\mu_i(v)=\mu_n(v)=0$ to pairwise disjoint intervals contained within the interval $(0,1)$.
			\item We assign all $v\in V$ with $\mu_i(v)=1$ and $\mu_n(v)=0$ to pairwise disjoint intervals contained within the interval $(1,2)$.
			\item We assign all $v\in V$ with $\mu_i(v)=2$ and $\mu_n(v)=0$ to the interval $[2,4]$.
			\item We assign all $v\in V$ with $\mu_i(v)=0$ and $\mu_n(v)=1$ to pairwise disjoint intervals contained within the interval $(3,4)$.
			\item We assign all $v\in V$ with $\mu_i(v)=1$ and $\mu_n(v)=1$ to the interval $[1,3]$.
			\item We assign all $v\in V$ with $\mu_i(v)=2$ and $\mu_n(v)=1$ to the interval $[0,4]$.
		\end{itemize}
		See \cref{fig:pid_1_2_box_n_1} for an illustration.
		\begin{figure}
			\centering
			\begin{tikzpicture}[decoration=brace, xscale = 0.8]
				\foreach \i in {1,2,4}
				{
					\draw[thick, |-|] (\i-0.5,2)--(\i,2);
					\draw[thick, |-|] (\i+4,2)--(\i+4.5,2);
					\draw[thick, |-|] (\i+13,0.5)--(\i+13.5,0.5);
				}
				\node at (2.75,2){$\dots$};
				\draw[thick, decorate] (0.3,2.2) to node[above, midway] {$\mu_i(v)=0$, $\mu_n(v)=0$} (4.2,2.2);
				\node at (7.25,2){$\dots$};
				\draw[thick, decorate] (4.8,2.2) to node[above, midway] {$\mu_i(v)=1$, $\mu_n(v)=0$} (8.7,2.2);
				\node at (16.25,0.5){$\dots$};
				\draw[thick, decorate] (13.8,0.7) to node[above, midway] {$\mu_i(v)=0$, $\mu_n(v)=1$} (17.7,0.7);
				\draw[thick, |-|] (9,2)to node[midway, above]{$\mu_i(v)=2$, $\mu_n(v)=0$}(18,2);
				\draw[thick, |-|] (4.5,0.5)to node[midway, above]{$\mu_i(v)=1$, $\mu_n(v)=1$}(13.5,0.5);
				\draw[thick, |-|] (0,-1)to node[midway, above]{$\mu_i(v)=2$, $\mu_n(v)=1$}(18,-1);
			\end{tikzpicture}
			\caption{Illustration of the construction of $H_i$ for $m_i=2$ in the proof of \cref{lemma:pid_1_2_box_n_1}.}\label{fig:pid_1_2_box_n_1}
		\end{figure}
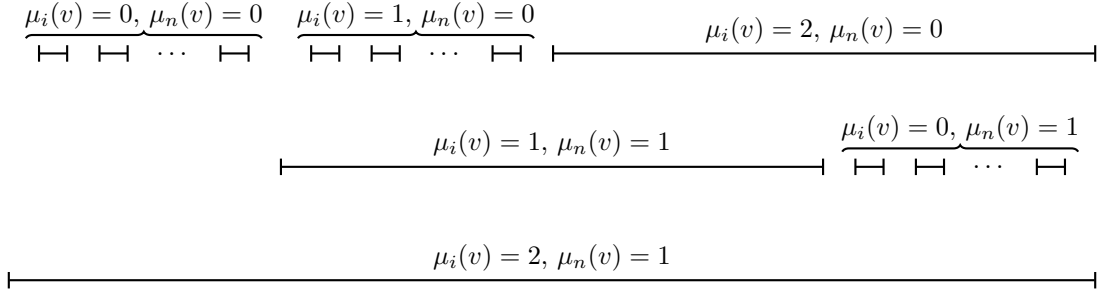
		We begin by showing that $H_i$ forms a supergraph of $G$ for every $i\in[n-1]$.
		\begin{claim}
			For every $i\in [n-1]$ with $m_i=1$, $H_i$ constitutes a supergraph of $G$.
		\end{claim}
		\begin{proof}[Proof of claim]
			Let $v,w\in V$ with $\{v,w\}\in E$. Then $\mu_i(v)=1$ or $\mu_i(w)=1$ and we may assume w.l.o.g.\ that $\mu_i(v)=1$. If further $\mu_n(v)=1$, then $\rho_i(v)=1$ and clearly $\{v,w\}\in E(H_i)$. Otherwise, we must have $\mu_n(w)=1$, so $\rho_i(v)+\rho_i(w)\ge \frac{2}{3}+\frac{1}{3}=1$. Again, $\{v,w\}\in E(H_i)$. 
		\end{proof}
		\begin{claim}
			For every $i\in [n-1]$ with $m_i=2$, $H_i$ constitutes a supergraph of $G$.
		\end{claim}
		\begin{proof}[Proof of claim]
			Let $v,w\in V$ with $\{v,w\}\in E$. Then $\mu_n(v)=1$ or $\mu_n(w)=1$ and we may assume w.l.o.g.\ that $\mu_n(v)=1$. If further $\mu_i(v)=2$, then the interval assigned to $v$ overlaps with all other intervals, so $\{v,w\}\in E(H_i)$. If $\mu_i(v)=1$, then we must have $\mu_i(w)\ge 1$. Moreover, the interval $[2,3]$ assigned to $v$ overlaps with all intervals assigned to elements $x$ with $\mu_i(x)\ge 1$, implying again $\{v,w\}\in E(H_i)$. Finally, if $\mu_i(v)=0$, we must have $\mu_i(w)=2$.  But then, the interval assigned to $w$ contains the interval $[3,4]$ and, in particular, the interval assigned to $v$. Once more, we can conclude that $\{v,w\}\in E(H_i)$.
		\end{proof}
		In order to show that the intersection of the graphs $(H_i)_{i=1}^{n-1}$ equals $G$, it remains to prove that for $v,w\in V$ with $\{v,w\}\notin E$, there exists $i\in [n-1]$ such that $\{v,w\}\notin E(H_i)$.
		
		First, suppose that there exists $i\in [n-1]$ such that $\mu_i(v)+\mu_i(w)< m_i$. In the case where $m_i=1$, this implies that $\mu_i(v)=\mu_i(w)=0$. But in this case, $\rho_i(v)+\rho_i(w)\le \frac{2}{3}< 1$, so $\{v,w\}\notin E(H_i)$.
		If $m_i=2$, then $\mu_i(r)=\mu_i(s)=0$, or $\{\mu_i(r),\mu_i(s)\}=\{0,1\}$. By observing that all intervals corresponding to vertices $x\in V$ with $\mu_i(x)=0$ only intersect intervals corresponding to vertices $y\in V$ with $\mu_i(y)=2$, we can conclude $\{v,w\}\notin E(H_i)$.
		
		Hence, we may assume in the following that $\mu_i(v)+\mu_i(w)\ge m_i$ for all $i\in [n-1]$. Moreover, as $\{v,w\}\notin E$, we must have $\mu_n(v)=\mu_n(w)=0$. From this we can derive that there must be $i\in [n-1]$ such that $\mu_i(v)<m_i$ or $\mu_i(w)<m_i$: Otherwise, we get $\mu_i(v)=\mu_i(w)=m_i$ for $i\in [n-1]$ and $\mu_n(v)=\mu_n(w)=0$, implying $v=w$ by the assumption of the lemma.
		
		So let $i\in [n-1]$ such that, w.l.o.g.\, $\mu_i(w)< m_i$. If $m_i=1$, then we get that $\mu_i(w)=0$, so $\rho_i(v)+\rho_i(w)\le \frac{2}{3}+\frac{1}{6}<1$, so $\{v,w\}\notin E(H_i)$. Finally, if $m_i=2$, then observing that an  interval corresponding to a vertex $x\in V$ with $\mu_i(x)\le 1$ and $\mu_n(x)=0$ do not intersect any interval corresponding to a vertex $y\in V$ with $\mu_n(y)=0$ yields $\{v,w\}\notin E(H_i)$.
	\end{proof}
	
	We remark that the graphs that we construct in the proof of \cref{lemma:box_th_1_2} were already used in \cite{2025_Chandran_ZeroDivisor} (see Lemma 2.5) to prove an upper bound on the boxicity of $\Gamma(\mathbb{Z}/N\mathbb{Z})$ in a setting that is almost identical to the one of \cref{lemma:box_th_1_2}.
	\begin{corollary}
		Let $n\ge 3$ and let $\bm{m}\in \{1,2\}^n\setminus \{\bm{2}_n\}$. Then $\boxicity(D(\bm{m}))=n-1$. Moreover, $\thdim(D(\bm{1}_n))=n-1$.
	\end{corollary}
	\begin{proof}
		The lower bound follows by \cref{prop:lower_bound_via_D_n}.
		The upper bound follows from \cref{lemma:box_th_1_2}.
	\end{proof}
	
	\begin{lemma}\label{lemma:box_D_3_1}
		Let $n\ge 2$ and let $\bm{m}\in \mathbb{N}^n$ such that $m_j=3$ for some $j\in [n]$ and $m_i = 1$ for all $i\in[n]\setminus\{j\}$. Then $\boxicity(D(\bm{m}))= n -1$.
	\end{lemma}
	\begin{proof}
		We may assume without loss of generality that $m_1=1$. Let $D(\bm{m})\eqqcolon (V,E)$.
		We construct interval graphs $(H_i)_{i=1}^{n-1}$ with vertex set $V$ whose intersection is $D(\bm{m})$.
		$H_1$ is defined as follows:
		\begin{itemize}
			\item We assign all $\bm{v}\in V$ with $v_1=v_n=0$ to pairwise disjoint intervals contained within $(0,1)$.
			\item We assign all $\bm{v}\in V$ with $v_1=0$ and $v_n=1$ to pairwise disjoint intervals contained within $(6,7)$.
			\item We assign all $\bm{v}\in V$ with $v_1= 1$ and $v_n=0$ to pairwise disjoint intervals contained within $(2,3)$.
			\item We assign all $\bm{v}\in V$ with $v_1=v_n=1$ to the interval $[4,5.5]$.
			\item We assign all $\bm{v}\in V$ with $v_1=2$ and $v_n=0$ to pairwise disjoint intervals contained within $(4,5)$. 
			\item We assign all $\bm{v}\in V$ with $v_1=2$ and $v_n=1$ to the interval $[2,5.5]$.
			\item We assign all $\bm{v}\in V$ with $v_1=3$ and $v_n=0$ to the interval $[5.5,7]$.
			\item We assign all $\bm{v}\in V$ with $v_1=3$ and $v_n=1$ to the interval $[0,7]$.
		\end{itemize}
		See \cref{fig:intervals_1_D_3_1} for an illustration.
		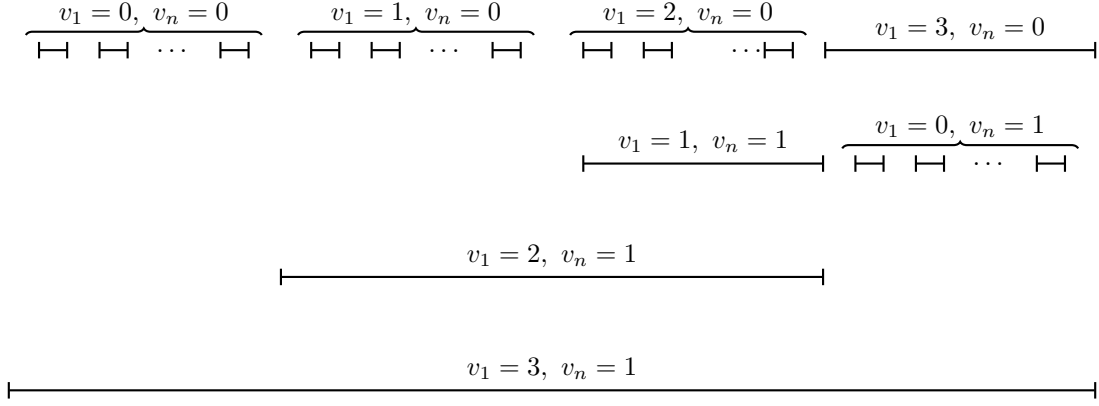
\begin{figure}
			\centering
			\begin{tikzpicture}[decoration=brace, xscale = 0.8]
				\foreach \i in {1,2,4}
				{
					\draw[thick, |-|] (\i-0.5,3.5)--(\i,3.5);
					\draw[thick, |-|] (\i+4,3.5)--(\i+4.5,3.5);
				}
				\node at (2.75,3.5){$\dots$};
				\draw[thick, decorate] (0.3,3.7) to node[above, midway] {$v_1=0,\ v_n=0$} (4.2,3.7);
				\node at (7.25,3.5){$\dots$};
				\draw[thick, decorate] (4.8,3.7) to node[above, midway] {$v_1=1,\ v_n=0$} (8.7,3.7);
				\foreach \i in {1,2,4}
				{
					\draw[thick, |-|] (\i+8.5,3.5)--(\i+9,3.5);
				}
				\node at (12.25,3.5){$\dots$};
				\draw[thick, decorate] (9.3,3.7) to node[above, midway] {$v_1=2,\ v_n=0$} (13.2,3.7);
				\draw[thick, |-|] (13.5,3.5)to node[midway, above]{$v_1=3,\ v_n=0$}(18,3.5);
				\foreach \i in {1,2,4}
				{
					\draw[thick, |-|] (\i+13,2)--(\i+13.5,2);
				}
				\node at (16.25,2){$\dots$};
				\draw[thick, decorate] (13.8,2.2) to node[above, midway] {$v_1=0,\ v_n=1$} (17.7,2.2);
				\draw[thick, |-|] (9.5,2)to node[midway, above]{$v_1=1,\ v_n=1$}(13.5,2);
				\draw[thick, |-|] (4.5,0.5)to node[midway, above]{$v_1=2,\ v_n=1$}(13.5,0.5);
				\draw[thick, |-|] (0,-1)to node[midway, above]{$v_1=3,\ v_n=1$}(18,-1);
			\end{tikzpicture}
			\caption{Illustration of the construction of $H_1$ in the proof of \cref{lemma:box_D_3_1}.}\label{fig:intervals_1_D_3_1}
		\end{figure}
		
		\begin{claim}
			The interval graph $H_1$ constitutes a supergraph of $D(\bm{m})$.
		\end{claim}
		\begin{proof}[Proof of claim]
			Let $\bm{v},\bm{w}\in V$ with $\{\bm{v},\bm{w}\}\in E$. Then $v_n=1$ or $w_n=1$ and we may assume w.l.o.g.\ that $v_n=1$.
			If $v_1=3$, the interval assigned to $v$ is $[0,7]$, which overlaps with every other interval, so $\{\bm{v},\bm{w}\}\in E(H_1)$.
			If $v_1=2$, then $w_1\geq 1$.
			The interval $[2,5.5]$ assigned to $\bm{v}$ contains all intervals assigned to vertices $\bm{x}$ with $x_1\in\{1,2\}$, shares the point $5.5$ with those with $x_n=0$ and $x_1=3$, and it is contained in the interval $[0,7]$ assigned to those with $x_n=1$ and $x_1=3$.
			Thus, $\{\bm{v},\bm{w}\}\in E(H_1)$.
			If $v_1=1$, then $w_1\geq 2$.
			The interval $[4,5.5]$ assigned to $\bm{v}$ is contained within all intervals assigned to vertices $\bm{x}$ with $x_1\in\{2,3\}$  and $x_n = 1$, contains all intervals assigned to vertices $\bm{x}$ with $x_1=2$ and $x_n=0$, and shares the point $5.5$ with those with $x_1=3$ and $x_n=0$.
			Thus, $\{\bm{v},\bm{w}\}\in E(H_1)$.
			Finally, if $v_1=0$, we must have $w_1=3$.
			Then the interval assigned to $\bm{v}$ is a contained in $(6,7)$ and the interval assigned to $\bm{w}$ is either $[0,7]$ or $[5.5,7]$.
			Thus, $\{\bm{v},\bm{w}\}\in E(H_1)$.
		\end{proof}
		
		For $i\in [n-1] \setminus \{1\}$, we define $H_i$ as follows:
		\begin{itemize}
			\item We assign all $\bm{v}\in V$ with $v_i=v_n=0$ to pairwise disjoint intervals contained within $(0,1)$ if $v_1\leq 1$, and within $(2,3)$ if $v_1\geq 2$.
			\item We assign all $\bm{v}\in V$ with $v_i=0$ and $v_n=1$ to pairwise disjoint intervals contained within $(4,5)$ if $v_1\geq 2$, and within $(6,7)$ if $v_1\leq 1$.
			\item We assign all $\bm{v}\in V$ with $v_i=1$ and $v_n=0$ to the interval $[4,7]$.
			\item We assign all $\bm{v}\in V$ with $v_i=1$, $v_n=1$, and $v_1 \leq 1$ to the interval $[2,5]$.
			\item We assign all $\bm{v}\in V$ with $v_i=1$, $v_n=1$, and $v_1 \geq 2$ to the interval $[0,7]$.
		\end{itemize}
		See \cref{fig:intervals_2_n-1_D_3_1} for an illustration.
		\begin{figure}
			\centering
			\begin{tikzpicture}[decoration=brace, xscale = 0.8]
				\foreach \i in {1,2,4}
				{
					\draw[thick, |-|] (\i-0.5,3.5)--(\i,3.5);
					\draw[thick, |-|] (\i+4,3.5)--(\i+4.5,3.5);
					\draw[thick, |-|] (\i+8.5,3.5)--(\i+9,3.5);
					\draw[thick, |-|] (\i+13,3.5)--(\i+13.5,3.5);
				}
				\node at (2.75,3.5){$\dots$};
				\draw[thick, decorate] (0.3,3.7) to node[above, midway] {\shortstack{$v_1\leq 1$\\  $v_i=0,\ v_n=0$}} (4.2,3.7);
				\node at (7.25,3.5){$\dots$};
				\draw[thick, decorate] (4.8,3.7) to node[above, midway] {\shortstack{ $v_1\geq 2$\\ $v_i=0,\ v_n=0$}} (8.7,3.7);
				\node at (11.75,3.5){$\dots$};
				\draw[thick, decorate] (9.3,3.7) to node[above, midway] {\shortstack{$v_1\geq 2$\\ $v_i=0,\ v_n=1$}} (13.2,3.7);
				\node at (16.25,3.5){$\dots$};
				\draw[thick, decorate] (13.8,3.7) to node[above, midway] {\shortstack{$v_1\leq 1$\\ $v_i=0,\ v_n=1$}} (17.7,3.7);
				\draw[thick, |-|] (9,2)to node[midway, above]{$v_i=1$, $v_n=0$}(18,2);
				\draw[thick, |-|] (4.5,0.5)to node[midway, above]{$v_i=1$, $v_n=1, v_1\leq 1$}(13.5,0.5);
				\draw[thick, |-|] (0,-1)to node[midway, above]{$v_i=1,\ v_n=1$, $v_1\geq 2$}(18,-1);
			\end{tikzpicture}
			\caption{Illustration of the construction of $H_i$ for $i \in [n-1] \setminus \{1\}$ in the proof of \cref{lemma:box_D_3_1}.}\label{fig:intervals_2_n-1_D_3_1}
		\end{figure}
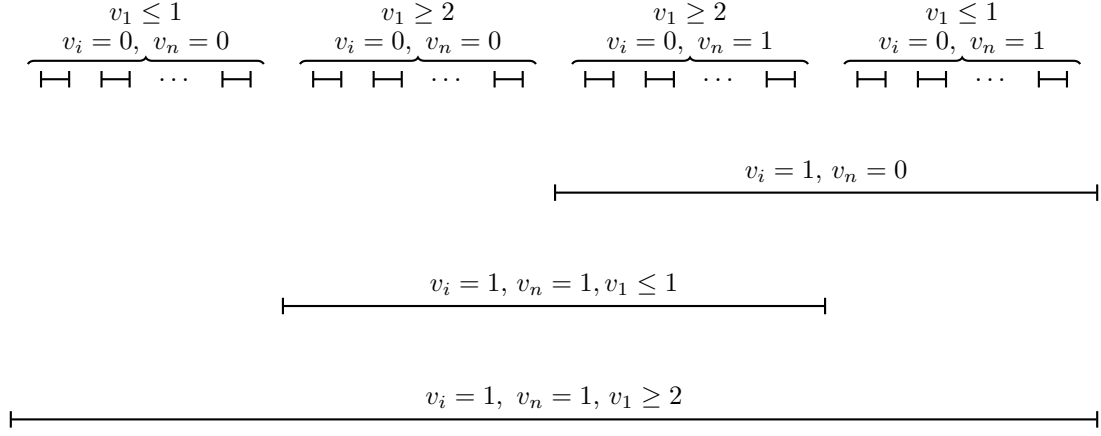
		\begin{claim}
			For every $i\in[n-1]\setminus\{1\}$, the interval graph $H_i$ constitutes a supergraph of $D(\bm{m})$.
		\end{claim}
		\begin{proof}[Proof of claim]
			Let $\bm{v},\bm{w}\in V$ with $\{\bm{v},\bm{w}\}\in E$. Then $v_n=1$ or $w_n=1$ and we may assume w.l.o.g.\ that $v_n=1$.
			If $v_i=1$ and $v_1\geq 2$, the interval assigned to $\bm{v}$ is $[0,7]$, 
			which contains every other interval, so $\{\bm{v},\bm{w}\}\in E(H_i)$.
			If $v_i=1$ and $v_1\leq 1$, then $w_1\geq 2$.
			The interval assigned to $\bm{v}$ is $[2,5]$ which overlaps with all intervals assigned to vertices $\bm{x}$ with $x_1\geq 2$, so $\{\bm{v},\bm{w}\}\in E(H_i)$.
			Finally, if $v_i=0$, then $w_i=1$.
			If, additionally, $v_1 \geq 2$, then the interval assigned to $\bm{v}$ is contained in $(4,5)$ and $\bm{w}$ is assigned to an interval containing $[4,5]$.
			Otherwise, $v_1 \le 1$ and so $w_1 \geq 2$.
			In this case, the interval assigned to $\bm{v}$ is contained in $(6,7)$ and $\bm{w}$ is assigned to an interval containing $[6,7]$.
			Thus, $\{\bm{v},\bm{w}\}\in E(H_i)$.
		\end{proof}
		In order to show that the intersection of the graphs $(H_i)_{i=1}^{n-1}$ equals $D(\bm{m})$, it remains to prove that for $\bm{v},\bm{w}\in V$ with $\{\bm{v},\bm{w}\}\notin E$, there exists $i\in [n-1]$ such that $\{\bm{v},\bm{w}\}\notin E(H_i)$.
		As $\{\bm{v},\bm{w}\}\notin E$, there exists $i\in [n]$ such that $v_i+w_i< m_i$.
		
		\vspace{2mm}
		\noindent
		\textbf{Case 1: $v_1 + w_1 < 3$.}

		\noindent
		We assume w.l.o.g.\ that $v_1 \leq w_1$.
		Then $3 > v_1 + w_1 \geq 2\cdot v_1$, so $v_1\leq 1$.
		If $v_1 = 0$, then $w_1\leq 2$.
		The vertices $\bm{x}$ with $x_1 = 0$ are mapped to pairwise disjoint intervals 
		contained in $(0,1) \cup (6,7)$ in the interval representation of $H_1$.
		These only overlap with intervals in $H_1$ assigned to vertices $\bm{x}$ with $x_1=3$.
		Hence,  $\{\bm{v},\bm{w}\}\notin E(H_1)$.
		If $v_1 = 1$, then $w_1 = 1$ because $v_1\le w_1< 3-v_1$.
		If, in addition, $v_n + w_n \leq 1$, then $\bm{v}$ and $\bm{w}$ are mapped to 
		disjoint intervals in $(2,3)$ if $v_n + w_n = 0$, or to a subset of 
		$(2,3)$ and to $[4,5.5]$ if $v_n + w_n = 1$.
		Either way, $\{\bm{v},\bm{w}\}\notin E(H_1)$.
		
		For the case $v_n = w_n = 1$, we rely on the interval graphs $H_i$ for 
		$i\in [n-1] \setminus \{1\}$ to delete $\{\bm{v},\bm{w}\}\notin E$.
		As $\bm{v} \neq \bm{w}$, there exists an index $j \in [n-1] \setminus \{1\}$ such that 
		$v_j\neq w_j$.
		As $m_j = 1$, we may assume w.l.o.g.\ that $v_j = 0$ and $w_j = 1$.
		Then, in the interval representation of $H_j$, $\bm{v}$ is mapped to a closed subinterval 
		of $(6,7)$ while $\bm{w}$ is mapped to the interval $[2,5]$.
		Thus, $\{\bm{v},\bm{w}\}\notin E(H_j)$.
		
		\vspace{2mm}
		\noindent
		\textbf{Case 2: $v_i = w_i = 0$ for $i \in [n-1] \setminus \{1\}$.}
		
		\noindent
		In the interval representation of $H_i$, the vertices $\bm{x}$ with $x_i = 0$ are mapped to pairwise disjoint intervals.
		Thus, $\{\bm{v},\bm{w}\}\notin E(H_i)$.
		
		\vspace{2mm}
		\noindent
		\textbf{Case 3: $v_n = w_n = 0$.}
		
		\noindent
		As $\bm{v} \neq \bm{w}$, there exists an index $j \in [n-1]$ such that $ v_j\neq w_j$.
		If $j = 1$, then $\{\bm{v},\bm{w}\} \notin E(H_1)$ because in the interval representation of $H_1$, every pair of vertices $\bm{x},\bm{y}$ with $x_n = y_n = 0$ and $x_1 \neq 3$ or $y_1 \neq 3$ are mapped to disjoint intervals.
		If $j \in [n-1] \setminus \{1\}$, we can assume w.l.o.g. that $v_j = 0$ and $w_j = 1$.
		Then, in the interval representation of $H_j$, $\bm{v}$ is mapped to a closed subinterval of $(0,3)$ while $\bm{w}$ is mapped to the interval $[4,7]$.
		Thus, $\{\bm{v},\bm{w}\}\notin E(H_j)$.
	\end{proof}
	\subsection{Lower bounds on boxicity and threshold dimension\label{sec:intergral_covering_lower_bounds}}
	\begin{lemma}\label{lemma:some_mi_at_least_4}
		Let $n\ge 2$ and let $\bm{m}\in\mathbb{N}^n_{\ge 1}$ such that there is $i\in [n]$ with $m_i\ge 4$. Then $\boxicity(D(\bm{m}))= \thdim(D(\bm{m}))= n$.
	\end{lemma}
	\begin{proof}
		Without loss of generality, we may assume that $m_n\ge 4$. By \cref{prop:generalized_disjointness_graph_upper_bound}, it suffices to prove that $\boxicity(D(\bm{m}))\ge n$. 
		To this end, we exhibit the $n$-fold join of $\overline{K_2}$ as an induced subgraph of $D(\bm{m})$.
		For $i\in [n-1]$, let $\bm{x^i}$ and $\bm{y^i}$ be given by
		\[x^i_j=\begin{cases}
			0 & j=i\\
			m_j & j\in [n]\setminus \{i\}
		\end{cases} \text{ and }y^i_j=\begin{cases}
			0 & j=i\\
			m_n-1 & j=n\\
			m_j & j\in [n-1]\setminus \{i\}
		\end{cases}. \]
		Let further $\bm{x^n}$ and $\bm{y^n}$ be given by
		\[x^n_j=\begin{cases}
			1 & j=n\\
			m_j & j\in [n-1]
		\end{cases} \text{ and }y^n_j=\begin{cases}
			2 & j=n\\
			m_j & j\in [n-1]
		\end{cases}. \]
		Note that the vertices $\bm{x^i},\bm{y^i},i\in [n]$ are valid vertices of $D(\bm{m})$ (in the sense that each of them has at least one non-zero coordinate and one coordinate that is less than the respective $m_j$). Moreover, they are pairwise distinct. For each $i\in [n]$, $\bm{x^i}$ and $\bm{y^i}$ are non-adjacent in $D(\bm{m})$. For $i\ne j\in [n]$, each of $\bm{x^i}$ and $\bm{y^i}$ is adjacent to each of $\bm{x^j}$ and $\bm{y^j}$. Hence, $D(\bm{m})[\{\bm{x^i},\bm{y^i}\colon i\in [n]\}]$ is isomorphic to the $n$-fold join of $\overline{K_2}$ as promised.
	\end{proof}
	\begin{lemma}\label{lemma:boxicity_all_at_least_2}
		Let $n\ge 2$ and $\bm{m}\in\mathbb{N}^n_{\ge 2}$. Then $\boxicity(D(\bm{m}))= \thdim(D(\bm{m}))= n$.
	\end{lemma}
	\begin{proof}
		By \cref{prop:generalized_disjointness_graph_upper_bound}, it suffices to prove that $\boxicity(D(\bm{m}))\ge n$.
		To this end, we exhibit the $n$-fold join of $\overline{K_2}$ as an induced subgraph of $D$.
		For $i\in [n]$, let $\bm{x^i}$ and $\bm{y^i}$ be given by 
		\[x^i_j=\begin{cases}
			0 & j=i\\
			m_j & j\in [n]\setminus \{i\}
		\end{cases} \text{ and }y^i_j=\begin{cases}
			1 & j=i\\
			m_j & j\in [n]\setminus \{i\}
		\end{cases}. \]
		Note that $n\ge 2$ and $m_i\ge 2$ for $i\in [n]$ ensures that each vertex has at least one coordinate which is non-zero and one coordinate  which is strictly smaller than the corresponding $m_i$; i.e., they are indeed vertices of $D(\bm{m})$.
		The vertices $\bm{x^i},\bm{y^i},i\in [n]$ are pairwise distinct. For each $i\in [n]$, $\bm{x^i}$ and $\bm{y^i}$ are non-adjacent in $D(\bm{m})$ because $x^i_i+y^i_i=1<m_i$. On the other hand, for $i\ne j\in [n]$, each of $\bm{x^i}$ and $\bm{y^i}$ is adjacent to each of $\bm{x^j}$ and $\bm{y^j}$. Hence, $D[\{\bm{x^i},\bm{y^i}\colon i\in [n]\}]$ is isomorphic to the $n$-fold join of $\overline{K_2}$ as promised.
	\end{proof}

	\begin{lemma}\label{lemma:thdim_at_least_one_2}
		Let $n\ge 2$ and let $\bm{m}\in\mathbb{N}^n_{\ge 1}\setminus\{\bm{1}_n\}$. Then $\thdim(D(\bm{m}))=n$.
	\end{lemma}
	\begin{proof}
		The upper bound follows from \cref{prop:generalized_disjointness_graph_upper_bound}. For the lower bound, we will show that $D(\bm{m})$ contains a partial $n$-fold join of $\overline{K_2}$ as an induced subgraph. We may assume without loss of generality that $m_n\ge 2$. 
		For $i\in [n-1]$, let $\bm{x^i}$ and $\bm{y^i}$ be given by
		\[x^i_j=\begin{cases}
			0 & j=i\\
			m_j & j\in [n]\setminus \{i\}
		\end{cases} \text{ and }y^i_j=\begin{cases}
			0 & j=i\\
			m_n-1 & j=n\\
			m_j & j\in [n-1]\setminus \{i\}
		\end{cases}. \]
		Let further $\bm{x^n}$ and $\bm{y^n}$ be given by
		\[x^n_j=\begin{cases}
			0 & j=n\\
			m_j & j\in [n-1]
		\end{cases} \text{ and }y^n_j=\begin{cases}
			m_n-1 & j=n\\
			m_j & j\in [n-1]
		\end{cases}. \]
		Note that the vertices $\bm{x^i},\bm{y^i},i\in [n]$ are valid vertices of $D(\bm{m})$ (in the sense that each of them has at least one non-zero coordinate and one coordinate that is less than the respective $m_j$). Moreover, they are pairwise distinct. For each $i\in [n]$, $\bm{x^i}$ and $\bm{y^i}$ are non-adjacent in $D(\bm{m})$. For $i\ne j\in [n-1]$, each of $\bm{x^i}$ and $\bm{y^i}$ is adjacent to each of $\bm{x^j}$ and $\bm{y^j}$. Moreover, for $i\in [n-1]$, $\bm{x^i}$ is adjacent to both $\bm{x^n}$ and $\bm{y^n}$, and $\bm{y^i}$ is adjacent to $\bm{y^n}$. So the only missing edge between the vertex sets $\{\bm{x^i},\bm{y^i}\}$ and $\{\bm{x^n},\bm{y^n}\}$ is $\{\bm{y^i},\bm{x^n}\}$. Hence, $D(\bm{m})[\{\bm{x^i},\bm{y^i}\colon i\in [n]\}]$ is isomorphic to an $n$-fold partial join of $\overline{K_2}$ as promised.
	\end{proof}
	\begin{proposition}\label{prop:C_4_and_d_partite_graph}
		Let $G=(V,E)$ be a graph and let $(H_i)_{i=1}^d$ be interval graphs on the vertex set $V$ with $E=\bigcap_{i=1}^d E(H_i)$.
		\begin{enumerate}[(a)]
			\item \label{item:induced_C_4} Let $a,b,c,d\in V$ be four distinct vertices such that $G[\{a,b,c,d\}]$ forms a complete bipartite graph $K_{2,2}$ with bipartitions $\{a,b\}$ and $\{c,d\}$, i.e., $\{a,b\},\{c,d\}\notin E$, but $\{v,w\}\in E$ for each $v\in\{a,b\}$ and $w\in\{c,d\}$. Then there is no $i\in [d]$ with $\{a,b\},\{c,d\}\notin E(H_i)$.
			\item \label{item:induced_d_partite_graph} Let $(A_i)_{i=1}^d$ be $d$ pairwise stable sets of $G$ such that $|A_i|\ge 2$ for each $i\in [d]$ and for $i\ne j\in [d]$ every vertex in $A_i$ is adjacent to every vertex in $A_j$. Then there exists a permutation $\pi\colon [d]\leftrightarrow [d]$ such that for $a,a'\in A_i$, we have $\{a,a'\}\notin E(H_j)$ if and only if $j=\pi(i)$. 
		\end{enumerate}
	\end{proposition}
	\begin{proof}
		\eqref{item:induced_C_4}: Suppose we had $\{a,b\},\{c,d\}\notin E(H_i)$ for some $i\in[d]$. As $E(G)\subseteq E(H_i)$, this implies that $H_i[\{a,b,c,d\}]\cong K_{2,2}\cong \overline{K_2}\vee \overline{K_2}$, contradicting the fact that $H_i$ is an interval graph.
		
		\eqref{item:induced_d_partite_graph}: For $i\in [d]$, let $x_i,y_i\in A_i$ be two distinct vertices. As $\{x_i,y_i\}\notin E(G)$, there is $\pi(i)\in [d]$ with $\{x_i,y_i\}\notin E(H_i)$. By \eqref{item:induced_C_4}, the indices $(\pi(i))_{i\in [d]}$ are pairwise distinct, i.e., $\pi$ is a permutation.
		Now, let $i\in [d]$ and let $a,a'\in A_i$ be two distinct vertices. Then $\{a,a'\}\notin E(G)$, so there is $k\in [d]$ with $\{a,a'\}\notin E(H_k)$. By \eqref{item:induced_C_4} and because $G[a,a',x_j,y_j]\cong K_{2,2}$ for every $j\ne i$, $k=\pi(i)$ is the only feasible option.
	\end{proof}
	\begin{lemma}\label{lemma:boxicity_3_2}
		Let $n\ge 3$ and let $\bm{m}\in\mathbb{N}^n_{\ge 1}$ such that there are distinct indices $i,j\in [n]$ with $m_i=3$ and $m_j\ge 2$. Then $\boxicity(D(\bm{m}))= \thdim(D(\bm{m}))= n$.
	\end{lemma}
	\begin{proof}
		By \cref{prop:generalized_disjointness_graph_upper_bound,prop:lower_bound_via_D_n,lemma:lower_bound_n_minus_1}, we know that $n-1\le \boxicity(D(\bm{m}))\le\thdim(D(\bm{m}))\le n$. As such, establishing that $\boxicity(D(\bm{m}))\ne n-1$ will prove the lemma.
		We may assume without loss of generality that $m_n=3$ and $m_{n-1}\ge 2$.
		Assume towards a contradiction that $\boxicity(D(\bm{m}))= n-1$. Then there exist $n-1$ interval graphs $(H_i)_{i=1}^{n-1}$ whose intersection is $D(\bm{m})\eqqcolon (V,E)$. For $i\in [n-2]$, let $\bm{a^i},\bm{b^i},\bm{c^i}\in V$ be given by \begin{equation}
			a^i_j=\begin{cases}
				0 & j = i\\
				m_j & j\in [n-2] \setminus \{i\}\\
				m_j & j= n-1\\
				3 & j=n
			\end{cases},
			b^i_j=\begin{cases}
				0 & j=i\\
				m_j & j\in [n-2]\setminus\{i\}\\
				m_j-1 & j=n-1\\
				3 & j=n\end{cases}
			\text{ and }
			c^i_j=\begin{cases}
				0 & j=i\\
				m_j & j\in[n-2]\setminus \{i\}\\
				m_j & j=n-1 \\
				2 & j=n\end{cases}.\label{eq:def_a_i_b_i_c_i}
		\end{equation}
		We further define $\bm{u},\bm{v}\in V$ by
		\begin{equation}
			u_j =\begin{cases}
				m_j & j\in [n-2]\\
				m_j-1 & j= n-1\\
				1 & j = n
			\end{cases} \text{ and }
			v_j=\begin{cases}
				m_j & j\in [n-2]\\
				m_j & j= n-1\\
				1 & j = n
			\end{cases}.
		\end{equation}
		Note that each of the vertices $\bm{u}$, $\bm{v}$, and $\bm{a^i},\bm{b^i},\bm{c^i}\in [n-2]$ has at least one non-zero coordinate and one coordinate that is strictly smaller than the corresponding $m_j$. As such, they are indeed valid vertices.
		By definition, they are pairwise distinct. For $i\in [n-2]$, let $A_i\coloneqq \{\bm{a^i},\bm{b^i},\bm{c^i}\}$ and let $A_{n-1}\coloneqq \{\bm{u},\bm{v}\}$. Then $(A_i)_{i\in [n-1]}$ are pairwise disjoint stable sets in $D(\bm{m})$. Moreover, for $i\ne j\in [n-1]$, every vertex in $A_i$ is adjacent to every vertex in $A_j$. By \cref{prop:C_4_and_d_partite_graph}~\eqref{item:induced_d_partite_graph}, there is $\pi\colon[d]\leftrightarrow[d]$ such that for distinct $a,a'\in A_i$, we have $\{a,a'\}\notin E(H_j)$ if and only if $j=\pi(i)$. By relabeling, we may assume without loss of generality that $\pi$ is the identity, that is,
		\begin{equation}
			\{\bm{a^i},\bm{b^i}\},\{\bm{a^i},\bm{c^i}\},\{\bm{b^i},\bm{c^i}\}\notin E(H_i)\text{ for every }i\in [n-2] \text{ and } \{\bm{u},\bm{v}\}\notin E(H_{n-1}). \label{eq:non_edges}
		\end{equation}
		Let $\bm{w},\bm{x},\bm{y}\in V$ be given by 
		\begin{equation}
			w_j =\begin{cases}
				m_j & j\in [n-2]\\
				m_j & j= n-1\\
				0 & j = n
			\end{cases}, x_j=\begin{cases}
				m_j & j\in[n-2]\\
				0 & j= n-1\\
				3 & j=n
			\end{cases} \text{ and } y_j=\begin{cases}
				m_j & j\in[n-2]\\
				1 & j= n-1\\
				3 & j=n
			\end{cases}.
		\end{equation}
		Observe that each of the vertices $\bm{w}$, $\bm{x}$ and $\bm{y}$ has at least one coordinate which is non-zero and at least one coordinate which is strictly smaller than the corresponding $m_j$. As such, they are indeed valid vertices of $D(\bm{m})$.
		Further note that $\bm{w}$, $\bm{x}$ and $\bm{y}$ are pairwise distinct and distinct from the previously defined vertices $\bm{u}$, $\bm{v}$ and $\bm{a^i},\bm{b^i},\bm{c^i},i\in [n-2]$.
		
		The vertices $\bm{v}$ and $\bm{w}$ are non-adjacent in $D(\bm{m})$, but each of them is adjacent to both $\bm{a^i}$ and $\bm{b^i}$ for $i\in [n-2]$. By \eqref{eq:non_edges} and \cref{prop:C_4_and_d_partite_graph}~\eqref{item:induced_C_4}, we must have
		\begin{equation}
			\{\bm{v},\bm{w}\}\notin E(H_{n-1}). \label{eq:non_edges_2}
		\end{equation} 
		Moreover, $\bm{x}$ and $\bm{y}$ are non-adjacent in $D(\bm{m})$, but each of them is adjacent to $\bm{a^i}$ and $\bm{c^i}$ for $i\in [n-2]$. By \eqref{eq:non_edges} and \cref{prop:C_4_and_d_partite_graph}~\eqref{item:induced_C_4}, we must have
		\begin{equation}
			\{\bm{x},\bm{y}\}\notin E(H_{n-1}). \label{eq:non_edges_3}
		\end{equation} 
		However, each of $\bm{v}$ and $\bm{w}$ is adjacent to each of $\bm{x}$ and $\bm{y}$, so \eqref{eq:non_edges_2} and \eqref{eq:non_edges_3} contradict \cref{prop:C_4_and_d_partite_graph}~\eqref{item:induced_C_4}.
	\end{proof}
	\subsection{Proof of \cref{theorem:boxicity_integral_covering}\label{sec:proof_boxicity_integral_covering}}
	For each of the cases, we list the auxiliary statements that give the respective bounds on the boxicity:
	\begin{itemize}
		\item $\exists i\in [n]\colon m_i \ge  4$: The lower bound follows from \cref{lemma:some_mi_at_least_4}, the upper bound from \cref{prop:generalized_disjointness_graph_upper_bound}.
		\item $\exists i\in[n] \colon m_i=3 \text{ and } \exists j\in [n]\setminus \{i\}\colon m_j\ge 2$: The upper bound follows from \cref{prop:generalized_disjointness_graph_upper_bound}. The lower bound follows from \cref{lemma:boxicity_3_2} for $n\ge 3$, and from \cref{lemma:boxicity_all_at_least_2} for $n=2$.
		\item $\exists i\in [n]\colon m_i=3 \text{ and } \forall j\in [n]\setminus \{i\}\colon m_j=1$: The lower bound follows from \cref{lemma:lower_bound_n_minus_1}. The upper bound follows from \cref{lemma:box_D_3_1}.
		\item $\bm{m}= \bm{2}_n$: The upper bound follows from \cref{prop:generalized_disjointness_graph_upper_bound}, the lower bound from \cref{lemma:boxicity_all_at_least_2}.
		\item $\bm{m}\in\{1,2\}^n \text{ and } \bm{m}\ne \bm{2}_n \text{ and } \bm{m}\ne \bm{1}_2$: The upper bound follows from \cref{lemma:box_th_1_2}. The lower bound follows from \cref{lemma:lower_bound_n_minus_1}.
		\item $\bm{m}=\bm{1}_2$: \cref{lemma:integral_covering_special_cases} yields the desired statement.
	\end{itemize}
	
	\subsection{Proof of \cref{theorem:thdim_integral_covering}\label{sec:proof_thdim_integral_covering}}
	For $\bm{m}=\bm{1}_2$, \cref{lemma:integral_covering_special_cases} yields the desired statement. For $n\ge 3$ and $\bm{m}=\bm{1}_n$, the upper bound follows from \cref{lemma:box_th_1_2}, while \cref{lemma:lower_bound_n_minus_1} yields the lower bound. For $m\ne \bm{1}_n$, the upper bound follows from \cref{prop:generalized_disjointness_graph_upper_bound}, and the lower bound follows from \cref{lemma:thdim_at_least_one_2}.

\subsection{Proof of \cref{theorem:boxicity_disjointness_graph}}\label{subsec:proof_box_disjoint}
\cref{theorem:boxicity_disjointness_graph} follows from \cref{prop:D_n_iso,lemma:integral_covering_special_cases,theorem:boxicity_integral_covering,theorem:thdim_integral_covering}.
	%
	\section{Zero divisor graphs of reduced rings}
	
	\subsection{From zero divisor graphs to disjointness graphs}
	In this section, we shed some light on the connection between zero divisor graphs and disjointness graphs, which will allow us to analyze the structure of zero divisor graphs from a combinatorial perspective.
	While this connection has already been pointed out in prior works, such as \cite{ANDERSON20121626,CHANDRAN2026127}, we give a self-contained derivation for the sake of completeness.
	
	Let $R$ be a reduced ring and let $\Min(R)=\{M_1,\dots,M_n\}$ denote its collection of minimal prime ideals.
	For $i\in [n]$, we define $\mu_i\colon Z(R)\setminus \{0\}\rightarrow \{0,1\}$ by setting $\mu_i(r)=1$ if $r\in M_i$, and $\mu_i(r)=0$ otherwise.
	That is, $\mu \colon Z(R)\setminus \{0\} \rightarrow \{0,1\}^n \setminus \{\mathbf{0},\mathbf{1}\}$ defined as $\mu(r) \coloneqq (\mu_i(r))_{i\in[n]}$ is the indicator vector of those minimal prime ideals containing $r$.
	\begin{lemma}\label{cor:zero_divisor_graph_reduced_threshold_graphs}
		Let $r,s\in Z(R)\setminus \{0\}$.
		It holds that $r\cdot s =0$ if and only if $\mu(r)+\mu(s)\ge \mathbf{1}$.
	\end{lemma}
	\begin{proof}
		By the definition of $\mu_i$, the condition $\mu_i(r)+\mu_i(s)\ge 1$ is equivalent to $r\in M_i$ or $s\in M_i$.
		First, suppose that $r\cdot s=0$. As every ideal contains $0$, we have $r\cdot s\in M_i$ for every $i\in [n]$.
		As each $M_i$ is prime, this implies $r\in M_i$ or $s\in M_i$ for every $i\in [n]$, as desired.
		Next, suppose that $r\in M_i$ or $s\in M_i$ for every $i\in [n]$. Then $r\cdot s\in M_i$ for every $i\in [n]$.
		Hence, $r\cdot s \in \bigcap \Min(R)$, and \cref{cor:reduced_intersection_min_prime_ideals} implies $r\cdot s=0$.
	\end{proof}
	Our next goal is to characterize the reduced zero divisor graph $\Gamma_E(R)$.
	To this end, we require the following lemma.
	\begin{lemma}\label{lemma:mu_surjective}
		The function $\mu$ is surjective, i.e. for every $\mathbf{v}\in \{0,1\}^n\setminus \{\mathbf{0},\mathbf{1}\}$, there is $r\in Z(R)\setminus \{0\}$ with $\mu(r) =\mathbf{v}$. 
	\end{lemma}
	\begin{proof}
		We first show that for every $i\in [n]$, there is $r_i\in Z(R)\setminus \{0\}$ with $\mu_i(r_i)=0$ but $\mu_j(r_i)=1$ for every $j\in [n]\setminus \{i\}$.
		This is equivalent to showing that $\bigcap_{j\in [n]\setminus \{i\}} M_j\not\subseteq M_i$.
		The latter follows from \cref{lemma:prime_ideal_containment} and the fact that distinct minimal prime ideals, by definition, do not contain each other.
		
		Next, pick $\mathbf{v}\in \{0,1\}^n\setminus \{\mathbf{0},\mathbf{1}\}$ and let $I\coloneqq \{i\in [n]\colon v_i=0\}$.
		Define $r\coloneqq \sum_{i\in I} r_i$.
		We show that $r\in M_i$ if and only if $v_i=1$.
		As $\mathbf{v}\notin\{\mathbf{0},\mathbf{1}\}$, this will imply $r\in Z(R)\setminus \{0\}$ by \cref{cor:union_of_minimal_ideals_reduced,cor:reduced_intersection_min_prime_ideals} and complete the proof.
		
		First, let $j\in [n]$ with $v_j=1$. Then $j\notin I$, so $r_i\in M_j$ for all $i\in I$. As ideals are closed under addition, $r\in M_j$.
		Next, let $j\in [n]$ with $v_j=0$. Suppose towards a contradiction that $r\in M_j$. Then $r_j=r-\sum_{i\in I\setminus \{j\}} r_i\in M_j$ because $r_i\in M_j$ for $i\in I\setminus \{j\}$ and ideals are closed under addition and taking additive inverses. But this contradicts our choice of $r_j$.
	\end{proof}
	We can now recover the following characterization of the reduced graph $\Gamma_E(R)$ from \cite[Theorem 1.1]{ANDERSON20121626}.
	\begin{corollary}\label{cor:disjointness_graph_subgraph}
		The reduced graph $\Gamma_E(R)$ is isomorphic to $D(\bm{1}_n)$.
	\end{corollary}
	\begin{proof}
		We show that for $r_1,r_2\in Z(R)\setminus \{0\}$, we have $r_1\sim r_2$ if and only if $\mu(r_1)=\mu(r_2)$.
		If $\mu(r_1)=\mu(r_2)$, then as $\mu(r_1)=\mu(r_2)\ne \bm{1}$, there is $i$ with $\mu_i(r_1)=\mu_i(r_2)=0$.
		By \cref{cor:zero_divisor_graph_reduced_threshold_graphs}, $r_1$ and $r_2$ are not adjacent in $\Gamma(R)$. Moreover, every $r\in Z(R)\setminus \{0\}$ is adjacent to $r_1$ if and only if it is adjacent to $r_2$. Hence, $r_1\sim r_2$.
		
		Conversely, suppose that $\mu(r_1)\ne \mu(r_2)$. W.l.o.g.\ assume there is $i$ with $\mu_i(r_1)=0$ but $\mu_i(r_2)=1$. Let $\bm{v}\in \{0,1\}^n$ with $v_i=0$ and $v_j=1$ for $j\ne i$. By \cref{lemma:mu_surjective}, let $r\in Z(R)\setminus \{0\}$ with $\mu(r)=\bm{v}$. Then $r$ is adjacent to $r_2$ but not to $r_1$, so $r_1\not\sim r_2$.
		
		Hence, \cref{lemma:mu_surjective,cor:zero_divisor_graph_reduced_threshold_graphs} tell us that $\mu$ induces an isomorphism between $\Gamma_E(R)$ and $D(\bm{1}_n)$.
	\end{proof}

	\subsection{Boxicity and threshold dimension}
	Using the connection between zero divisor graphs, threshold graphs and integral covering graphs established in the previous section, we are now ready to exactly determine the boxicity and threshold dimension of zero divisor graphs. To state our result, we require the following definition.
	\begin{definition}
		We call an index $i\in [n]$ \emph{lonely} if there exists a unique $r\in R$ such that $\mu_i(r)=0$ and $\mu_j(r)=1$ for all $j\in [n]\setminus \{i\}$.
	\end{definition}
	Note that $i\in [n]$ is lonely if and only if $|\bigcap_{j\in [n]\setminus \{i\}} M_j\setminus M_i|=1$.
	We restate \cref{thm:box_thd_reduced} for convenience.
	\thmBoxThdReduced*
	\begin{proof}
		If $n=1$, then $Z(R)=\{0\}$, and, hence, $\Gamma(R)=\emptyset$ by \cref{cor:reduced_intersection_min_prime_ideals,cor:union_of_minimal_ideals_reduced}.
		
		Next, let $n\ge 2$. If there is no lonely index, then, for every $i\in [n]$, there exist $r_i\ne s_i\in Z(R)\setminus \{0\}$ such that $\mu_i(r_i)=\mu_i(s_i)=0$ and $\mu_j(r_i)=\mu_j(s_j)=1$ for all $j\in [n]\setminus \{i\}$. By \cref{cor:zero_divisor_graph_reduced_threshold_graphs}, $\Gamma(R)[\{r_i,s_i\colon i\in [n]\}]$ forms an $n$-fold join of $\overline{K_2}$. This implies $n\le \boxicity(\Gamma(R))\le \thdim(\Gamma(R))$. The matching upper bound follows from \cref{cor:zero_divisor_graph_reduced_threshold_graphs}.
		
		If there is a lonely index, then \cref{lemma:box_th_1_2} implies that $\thdim(\Gamma(R))\le n-1$. For $n\ge 3$, \cref{theorem:boxicity_integral_covering,cor:disjointness_graph_subgraph} yield the matching lower bound.
		
		Finally, for $n=2$, if there is an index that is not lonely, then the same arguments as before yield a subgraph of $\Gamma(R)$ isomorphic to $\overline{K_2}$, giving $1\le \boxicity(\Gamma(R))\le \thdim(\Gamma(R))\le 1$. On the other hand, if both $1$ and $2$ are lonely indices, then $\Gamma(R)\cong D(\bm{1}_2)$ has boxicity and threshold dimension $0$ by \cref{theorem:boxicity_integral_covering}.
	\end{proof}
	
	\section{Quotient rings of principal ideal domains}
	\subsection{Principal ideal domains}
For this section, the term `ring' refers to a commutative, non-zero ring with $1$.
	\begin{definition}
		A ring $R$ is called a \emph{domain} if $\{0\}$ is a prime ideal, i.e., if $r\cdot s=0$ implies $r=0$ or $s=0$.
	\end{definition}
	Note that a ring is a domain if and only if $Z(R)=\{0\}$.
	\begin{definition}
		A domain $R$ is called a \emph{principal ideal domain (pid)} if every ideal $I\subseteq R$ is of the form $I=a\cdot R=\{a\cdot r\colon r\in R\}$.
	\end{definition}
	Examples for pids include the ring of integers or polynomial rings of the form $K[x]$, where $K$ is a field. Like the integers, principal ideal domains permit a factorization into \emph{prime elements}, which is unique up to multiplication with \emph{units}.
	\begin{definition}
		Let $R$ be a ring. We call $p\in R$ a \emph{prime element} if $p\cdot R$ is a prime ideal of $R$. We call $u\in R$ a \emph{unit} if there exists $v\in R$ such that $u\cdot v=1$. We denote the set of units of $R$ by $R^\times$.
	\end{definition}
	Given a ring $R$, we can define an equivalence relation $\sim$ on $R$ by letting $r\sim s$ if and only if there exists a unit $u\in R^\times$ such that $r=u\cdot s$. Note that if $r\sim s$, then $r$ is a prime element if and only if $s$ is. Two prime elements $r$ and $s$ with $r\sim s$ are called \emph{associated}.
	\begin{definition}[prime factorization]
		Let $R$ be a pid and let $r\in R\setminus \{0\}$. A \emph{prime factorization of $r$} consists of a unit $u\in R^\times$, $n\in\mathbb{N}_0$, pairwise non-associated primes $(p_i)_{i=1}^n$ and exponents $(m_i)_{i=1}^n\in \mathbb{Z}_{>0}^n$ such that $r=u\cdot\prod_{i=1}^n p_i^{m_i}$, where $\prod \emptyset \coloneqq 1$.
	\end{definition}
	
	\begin{theorem}[unique prime factorization]\label{theorem:prime_factorization}
		Let $R$ be a pid and let $r\in R\setminus \{0\}$. Then $r$ admits a prime factorization, which is unique in the following sense:
		Suppose that $r=u\cdot\prod_{i=1}^n p_i^{m_i}$ and $r=u'\cdot \prod_{i=1}^{n'}{p'_i}^{m'_i}$ are two prime factorizations of $r$. Then $n=n'$ and there exists a bijection $\varphi\colon [n]\leftrightarrow [n]$ such that $p'_{\varphi(i)}\sim p_i$ and $m'_{\varphi(i)}=m_i$ for all $i\in [n]$.
	\end{theorem}
	\begin{definition}
		Let $r,s\in R$. We write $r\mid s$ if there exists $t\in R$ with $r\cdot t=s$.
	\end{definition}
	Note that $1\mid r$ and $r\mid 0$ for every $r\in R$.
	\begin{proposition}\label{prop:division}
		Let $r,s\in R\setminus \{0\}$. Let $r=u\cdot \prod_{i=1}^n p_i^{m_i}$ and $s=u'\cdot \prod_{i=1}^{n'} p_i^{m'_i}$ be prime factorizations of $r$ and $s$.
		Then $r\mid s$ if and only if for every $i\in [n]$, there is $j\in [n']$ with $p_i\sim p'_j$ and $m_i\le m'_j$. In particular, $r\mid s$ if and only if $p_i^{m_i}\mid s$ for every $i\in [n]$.
	\end{proposition}
	\subsection{Quotient rings}
	Let $R$ be a ring and let $I\subsetneq R$ be a proper ideal of $R$. We define an equivalence relation $\sim_I$ on $R$ by setting $r\sim_I s$ if and only if $r-s\in I$. Let $R/I$ be the set of equivalent classes of $\sim_I$ and let $\pi\colon R\rightarrow R/I$ be the projection map sending $r\in R$ to its equivalence class. Then there is a unique ring structure on $R/I$ that turns $\pi$ into a ring homomorphism. We call $R/I$, equipped with this ring structure, a \emph{quotient ring} of $R$.
	\begin{proposition}
		Let $R$ be a ring and let $I\subseteq R$ be an ideal. Then $R/I$ is a domain if and only if $I$ is a prime ideal.
	\end{proposition}
	\subsection{Zero divisor graphs of quotient rings of pids}
	In this section, we will analyze the structure of zero divisor graphs of finite quotient rings of pids. Note that these include rings of the form $\mathbb{Z}/N\mathbb{Z}$ for $N\in\mathbb{N}_{\ge 2}$, studied in prior works~\cite{2025_Chandran_ZeroDivisor,2026_Chandran_Compressed}.
	For the remainder of this section, let $R$ be a pid and let $I=g\cdot R\subsetneq R$ be a proper ideal such that $Q\coloneqq R/I$ is finite. We will assume that $I$ is not prime; otherwise $Q$ will not have any zero divisors except $0$. By \cref{theorem:prime_factorization}, there exist a unit $u$, pairwise non-equivalent primes $(p_i)_{i=1}^n$ and positive exponents $(m_i)_{i=1}^n$ such that $g=u\cdot \prod_{i=1}^n p_i^{m_i}$. The fact that $I$ is proper and not prime implies $\sum_{i=1}^n m_i\ge 2$.
	In order to analyze the zero divisor graph of $Q$, we need to introduce some notation.
	For $r\in R\setminus \{0\}$ and $i\in [n]$, we define $a_i(r)\coloneqq \max\{a\in\mathbb{N}_0\colon p_i^a \mid r\}$. Note that this maximum exists by \cref{prop:division}.
	
We remark that this function has already been considered in~\cite{2025_Chandran_ZeroDivisor}.
	
	\begin{proposition}\label{prop:exponents_multiplication}
		Let $r,s\in R\setminus \{0\}$ and $i\in [n]$. Then $a_i(r\cdot s)=a_i(r)+a_i(s)$.
	\end{proposition}
	
	\begin{proposition}\label{prop:zero_product_in_quotient}
		Let $r,s\in R\setminus \{0\}$. Then $\pi(r\cdot s)=\pi(r)\cdot \pi(s)=0$ if and only if for every $i\in [n]$, $a_i(r)+a_i(s)\ge m_i$.
	\end{proposition}
	\begin{proof}
		We have $\pi(r\cdot s)=\pi(r)\cdot \pi(s)=0$ if and only if $r\cdot s\in I$, which is equivalent to $g\mid r\cdot s$. Hence, \cref{prop:division,prop:exponents_multiplication} yield the desired statement.
	\end{proof}
	Motivated by \cref{prop:zero_product_in_quotient}, for $q=\pi(r)\in Q\setminus \{0\}$, we define $\alpha(q)=(\alpha_i(q))_{i=1}^n\in \prod_{i=1}^n \{0,\dots,m_i\}$ by setting $\alpha_i(q)=\min\{a_i(r),m_i\}$. \cref{prop:exponents_well_defined} tells us that $\alpha(q)$ is indeed well-defined, meaning that it does not depend on the choice of the representative $r$.
	\begin{proposition}\label{prop:exponents_well_defined}
		Let $r,s\in R\setminus\{0\}$ with $\pi(r)=\pi(s)$. Then for every $i\in [n]$, $\min\{a_i(r),m_i\}=\min\{a_i(s),m_i\}$.
	\end{proposition}
	\begin{proof}
		As $r-s\in I=g\cdot R$, we have $p_i^{m_i}\mid r-s$. Hence, for every $t\le m_i$, $p_i^t\mid r$ if and only if $p_i^t\mid s$.
	\end{proof}
	\cref{prop:zero_product_in_quotient,prop:exponents_well_defined} tell us that $\Gamma(Q)$ can be expressed as the intersection of $n$ threshold graphs.
	\begin{corollary}\label{cor:zero_divisor_quotient_threshold}
		$\Gamma(Q)$ is the intersection of the threshold graphs $T(\alpha_i\upharpoonright_{Z(Q)\setminus \{0\}},m_i)_{i\in [n]}$, where $\upharpoonright_{Z(Q)\setminus \{0\}}$ denotes the restriction of a function to the ground set $Z(Q)\setminus \{0\}$.
	\end{corollary}
	We further observe that the function $\alpha$ is surjective. 
	\begin{proposition}\label{prop:alpha_surjective}
		For every $\alpha\in \prod_{i=1}^n [m_i]_0\setminus\{0,(m_i)_{i=1}^n\}$, there exists $q\in Z(Q)\setminus \{0\}$ with $\alpha(q)=\alpha$.
	\end{proposition}
	\begin{proof}
		Let $r\coloneqq \prod_{i=1}^n p_i^{\alpha_i}$ and $q\coloneqq \pi(r)$. The desired statement follows from \cref{prop:division,prop:zero_product_in_quotient}.
	\end{proof}
	
	\begin{proposition}\label{prop:zero_divisors_quotient_ring}
	For $q\in Q\setminus \{0\}$, we have $\alpha(q)\ne (m_i)_{i=1}^n$. In particular,
	\[Z(Q)\setminus \{0\}\coloneqq \{q\in Q\setminus \{0\}\colon \alpha(q)\neq \bm{0}_n\}=\{q\in Q\setminus \{0\}\colon \alpha(q)\notin \{\bm{0}_n,(m_i)_{i=1}^n\}\}.\]
	\end{proposition}
	\begin{proof}
	For the first part, we observe that if $q=\pi(r)$ with $a_i(r)\ge m_i$ for every $i\in [n]$, then $r\in I$ and $q=0$. The second part then follows from \cref{prop:zero_product_in_quotient,prop:alpha_surjective}.
	\end{proof}

	We conclude this section by relating the boxicity of the reduced zero divisor graph of $Q$ to the boxicity of an integral covering graph (cf. \cref{thm:boxicity_two_equiv_relations}). Let $\bm{m}\coloneqq (m_i)_{i\in [n]}$.
	\begin{corollary}\label{cor:reduced_zero_divisor_graph_D_m}
		For $n\ge 3$, the graph $\Gamma_{E'}(Q)$ is isomorphic to $D(\bm{m})$.
	\end{corollary}
	\begin{proof}
		By \cref{prop:alpha_surjective}, for every $\bm{v}\in V(D(\bm{m}))$, there is $q\in Z(Q)\setminus \{0\}$ with $\alpha(q)=\bm{v}$. Moreover, \cref{prop:zero_product_in_quotient,prop:exponents_well_defined} tell us that $q,q'\in Z(Q)\setminus \{0\}$ with $\alpha(q)=\alpha(q')$ have the same neighborhood in $Z(Q)\setminus\{0,q,q'\}$. On the other hand, if $\alpha(q)\ne \alpha(q')$, then let $i\in [n]$ with $\alpha_i(q)\ne \alpha_i(q')$. We may assume w.l.o.g.\ that $\alpha_i(q)>\alpha_i(q')$. Let $\bm{v}\in V(D(\bm{m}))$ be given by $v_i=m_i-\alpha_i(q)$ and $v_j=m_j$ for $j\in [n]\setminus \{i\}$ and let $\bm{w}=\bm{m}-\alpha(q)\in V(D(\bm{m}))$. If $\bm{v}\notin\{\alpha(q),\alpha(q')\}$, then $r\in Z(Q)\setminus \{q,q',0\}$ with $\alpha(r)=\bm{v}$ is adjacent to $q$, but not to $q'$. Next, if $\bm{v}=\alpha(q)$, then $\bm{w}\ne \alpha(q)$ because $\alpha_j(q)=m_j\ne 0=w_j$ for $j\in [n]\setminus \{i\}$. Moreover, $\bm{w}\ne \alpha(q')$ because $\alpha_i(q')<\alpha_i(q)=v_i=w_i$. Hence, $r\in Z(Q)\setminus\{0,q,q'\}$ with $\alpha(r)=\bm{w}$ is adjacent to $q$, but not to $q'$.
		Finally, assume that $\alpha(q')=\bm{v}$. If there is $j\in [n]\setminus \{i\}$ with $\alpha_j(q)<m_j$, then let $k\in [n]\setminus \{i,j\}$ and let $\bm{x},\bm{y}\in V(D(\bm{m}))$ with $x_i=y_i=m_i$, $x_\ell=y_\ell=0$ for $\ell\in [n]\setminus \{i,k\}$, $x_k=0$ and $y_k=m_k$. Then $\alpha(q')\notin \{\bm{x},\bm{y}\}$ because $\alpha_j(q')=v_j=m_j$. Let $\bm{z}\in \{\bm{x},\bm{y}\}\setminus\{\alpha(q)\}$ and let $r\in Z(Q)\setminus \{0,q,q'\}$ with $\alpha(r)=\bm{z}$. Then $r$ is adjacent to $q'$, but not to $q$ in $\Gamma(Q)$.
		On the other hand, if $\alpha_j(q)=\alpha_j(q')=m_j$ for every $j\in [n]\setminus \{i\}$, then $\bm{w}\notin \{\alpha(q),\alpha(q')\}$ and $r\in Z(Q)\setminus \{0,q,q'\}$ with $\alpha(r)=\bm{w}$ is adjacent to $q$, but not to $q'$ in $\Gamma(Q)$.
	\end{proof}
	\begin{corollary}\label{cor:reduced_zero_divisor_graph_D_m_n_2}
	For $n=2$ and $\bm{m}\ne \bm{1}_2$, the graph $\Gamma_{E'}(Q)$ is isomorphic to $D(\bm{m})$.
	\end{corollary}
	\begin{proof}
	By \cref{prop:alpha_surjective}, for every $\bm{v}\in V(D(\bm{m}))$, there is $q\in Z(Q)\setminus \{0\}$ with $\alpha(q)=\bm{v}$. Moreover, \cref{prop:zero_product_in_quotient,prop:exponents_well_defined} tell us that $q,q'\in Z(Q)\setminus \{0\}$ with $\alpha(q)=\alpha(q')$ have the same neighborhood in $Z(Q)\setminus\{0,q,q'\}$. 
	It remains to prove that $q,q'\in Z(Q)\setminus \{0\}$ with $\alpha(q)\ne\alpha(q')$ do not have the same neighborhood in $Z(Q)\setminus\{0,q,q'\}$. To this end, using $\bm{m}\ne \bm{1}_2$, we will assume without loss of generality that $m_2\ge 2$; the case $m_1\ge 2$ can be handled analogously.
	Let $q,q'\in Z(Q)\setminus \{0\}$ with $\alpha(q)\ne\alpha(q')$. We distinguish two cases:\\
	\textbf{Case 1:} $\alpha_1(q)\ne \alpha_1(q')$. We may assume without loss of generality that $\alpha_1(q')<\alpha_1(q)$. In particular, $\alpha_1(q)\ge 1$ and $\alpha_1(q')\le m_1-1$.
	We distinguish two subcases:\\
	\textbf{Case 1.1:} $\alpha_2(q)\ge 1$. Consider the two vectors $\bm{v}\coloneqq (m_1-\alpha_1(q),m_2)$ and $\bm{w}\coloneqq (m_1-\alpha_1(q),m_2-1)$. Then $\bm{v},\bm{w}\in V(D(\bm{m}))$ because the first component is strictly smaller than $m_1$, while the second component is strictly larger than $0$. We further observe that $\{\bm{v},\bm{w}\}\setminus \{\alpha(q),\alpha(q')\}\ne \emptyset$. Indeed, as $v_1=w_1$, but $\alpha_1(q)\ne \alpha_1(q')$, the shared first component of $\bm{v}$ and $\bm{w}$ can agree with at most one of $\alpha_1(q)$ and $\alpha_1(q')$. But then, as $v_2\ne w_2$, at most one of $\bm{v}$ and $\bm{w}$ can coincide with the respective vector in $\{\alpha(q),\alpha(q')\}$.
	
	Let $\bm{x}\in \{\bm{v},\bm{w}\}\setminus \{\alpha(q),\alpha(q')\}$. By \cref{prop:alpha_surjective}, let $q_x\in Z(Q)\setminus \{0\}$ with $\alpha(q_x)=\bm{x}$. Then $q_x\notin\{q,q'\}$. Moreover, by \cref{prop:zero_product_in_quotient}, $q_x$ is adjacent to $q$, but not to $q'$ in $\Gamma(Q)$. Hence, $q$ and $q'$ do not have the same neighborhood in $Z(Q)\setminus\{0,q,q'\}$.\\
	\textbf{Case 1.2:} $\alpha_2(q)=0$. If $\alpha_2(q')=0$ as well, then consider $\bm{v}\coloneqq (m_1-\alpha_1(q),m_2)$ and, by \cref{prop:alpha_surjective}, pick $q_v\in Z(Q)\setminus \{0\}$ with $\alpha(q_v)=\bm{v}$. Then $q_v\notin \{q,q'\}$ and $q_v$ is adjacent to $q$, but not to $q'$. If $\alpha_2(q')\ge 1$, then consider $\bm{u}=(m_1,m_2-1)$, and, by \cref{prop:alpha_surjective}, pick $q_u\in Z(Q)\setminus \{0\}$ with $\alpha(q_u)=\bm{u}$. Then $q_u\notin \{q,q'\}$ and $q_u$ is adjacent to $q'$, but not to $q$. Hence, $q$ and $q'$ do not have the same neighborhood in $Z(Q)\setminus\{0,q,q'\}$.\\
	\textbf{Case 2:} $\alpha_1(q)=\alpha_1(q')$. Then $\alpha_2(q)\ne \alpha_2(q')$, assume without loss of generality that $\alpha_2(q')<\alpha_2(q)$. In particular, $\alpha_2(q)\ge 1$.
	Again, we distinguish two subcases.\\
	\textbf{Case 2.1:} $\alpha_1(q)=\alpha_1(q')\le m_1-1$. Let $\bm{v}\coloneqq (m_1,m_2-\alpha_2(q))$, and, by \cref{prop:alpha_surjective}, pick $q_v\in Z(Q)\setminus \{0\}$ with $\alpha(q_v)=\bm{v}$. Then $q_v\notin \{q,q'\}$ and $q_v$ is adjacent to $q$, but not to $q'$. Hence, $q$ and $q'$ do not have the same neighborhood in $Z(Q)\setminus\{0,q,q'\}$.\\
	\textbf{Case 2.2:} $\alpha_1(q)=\alpha_1(q')=m_1$. By \cref{prop:zero_divisors_quotient_ring}, $\alpha_2(q)\le m_2-1$.
	Let $\bm{v}\coloneqq (0,m_2-\alpha_2(q))$, and, by \cref{prop:alpha_surjective}, pick $q_v\in Z(Q)\setminus \{0\}$ with $\alpha(q_v)=\bm{v}$. Again, $q_v\notin \{q,q'\}$ and $q_v$ is adjacent to $q$, but not to $q'$. Hence, $q$ and $q'$ do not have the same neighborhood in $Z(Q)\setminus\{0,q,q'\}$.
	\end{proof}
	\begin{corollary}\label{cor:reduced_zero_divisor_graph_D_m_1_1}
	Let $\bm{m}=\bm{1}_2$. Then $\Gamma_{E'}(Q)$ is a clique, and in particular, $\boxicity(\Gamma_{E'}(Q))=0$.
	\end{corollary}
	\begin{proof}
	By \cref{prop:zero_product_in_quotient,prop:zero_divisors_quotient_ring}, $\Gamma_{E'}(Q)$ has at most two vertices, corresponding to the equivalence classes of zero divisors $q\in Z(Q)\setminus \{0\}$ with $\alpha(q)=(0,1)$ and $\alpha(q)=(1,0)$, respectively. As $(0,1)+(1,0)=\bm{1}_2$ these two vertices, if distinct, are adjacent. 
	\end{proof}
	\subsection{Boxicity and threshold dimension}
	In this section, we analyze the boxicity and threshold dimension of $\Gamma(Q)$. We begin by proving the following result, which tells us that the threshold dimension and boxicity of $\Gamma(Q)$ must be either $n-1$ or $n$. The remainder of this section will be dedicated to providing a full distinction between these two cases.
	\begin{lemma}
		We have $n-1\le \boxicity(\Gamma(Q))\le \thdim(\Gamma(Q))$ for $n\ge 3$ and $\boxicity(\Gamma(Q))\le \thdim(\Gamma(Q))\le n$ for every $n$.\label{lemma:simple_bounds_quotient}
	\end{lemma}
	\begin{proof}
		For $n\ge 3$, define $\iota\colon 2^{[n]}\setminus\{\emptyset,[n]\}\rightarrow Z(Q)\setminus \{0\}$ by setting $\iota(I)\coloneqq \pi(\prod_{i\in [n]\setminus I} p_i^{m_i})$. Then $\iota$ is injective and by \cref{prop:zero_product_in_quotient}, we have $\iota(I)\cdot \iota(J)=0$ if and only if $I\cap J=\emptyset$. Hence, $\Gamma(Q)[\iota(2^{[n]}\setminus\{\emptyset,[n]\})]\cong D_n$. This implies the lower bound by \cref{theorem:boxicity_disjointness_graph}.
		
		For the upper bound, we observe that \cref{prop:zero_product_in_quotient} implies $\Gamma(Q)$ can be written as the intersection of the $n$ threshold graphs $T(\alpha_i,m_i)_{i=1}^n$.
	\end{proof}
	We begin by settling the case $n=1$.
	
	\begin{lemma}\label{lemma:n_eq_1}
		Let $n=1$. Then \[\thdim(\Gamma(Q))=\boxicity(\Gamma(Q))=\begin{cases}
			0 & m_1 \le 2\\
			1 & m_1\ge 3
		\end{cases}.\]
	\end{lemma}
	\begin{proof}
		By \cref{prop:zero_divisors_quotient_ring}, we have $\alpha(q)\in \{1,\dots,m_1-1\}$ for every $q\in Z(Q)\setminus \{0\}$ and by \cref{prop:zero_product_in_quotient}, we have $q_1\cdot q_2=0$ for $q_1,q_2\in Z(q)\setminus \{0\}$ if and only if $\alpha(q_1)+\alpha(q_2)\ge m_1$. For $m_1=1$, $\Gamma(Q)$ is empty, and for $m_1=2$, it is a clique, so $\thdim(\Gamma(Q))=\boxicity(\Gamma(Q))=0$ in this case. For $m_1\ge 3$, we have $q_1\coloneqq\pi(p_1)\ne \pi(p_1+p_1^2)\eqqcolon q_2$ because $p_1^2\notin p_1^{m_1}\cdot R$. Moreover, $\alpha(q_1)=\alpha(q_2)=1$, so $\{q_1,q_2\}$ forms an independent set of size $2$ in $\Gamma(Q)$. By \cref{lemma:simple_bounds_quotient}, we have $1= \boxicity(\Gamma(Q))\le \thdim(\Gamma(Q))$ in this case. 
	\end{proof}
	
	\begin{lemma}\label{lemma:n_eq_1_reduced}
	Let $n=1$. Then \[\thdim(\Gamma_E(Q))=\boxicity(\Gamma_E(Q))=\begin{cases}
		0 & m_1 \le 3\\
		1 & m_1\ge 4
	\end{cases}.\]
	\end{lemma}
	\begin{proof}
	First, let $m_1\le 3$. We show that $\Gamma_{E}(Q)$ is a clique. To this end, it suffices to show that whenever $q_1,q_2\in Z(Q)\setminus \{0\}$ with $q_1\cdot q_2\ne 0$, then $q_1\sim q_2$. Let $q_1,q_2\in Z(Q)\setminus \{0\}$ with $q_1\cdot q_2\ne 0$. By \cref{prop:zero_divisors_quotient_ring}, we know that $1\le \alpha(q_1),\alpha(q_2)\le m_1-1$. By \cref{prop:zero_product_in_quotient}, we can infer that $\alpha(q_1)+\alpha(q_2)<m_1\le 3$, so we must have $\alpha(q_1)=\alpha(q_2)=1$ and $m_1=3$. But then, again by \cref{prop:zero_product_in_quotient}, $q_1$ and $q_2$ are not adjacent and have the same neighborhood in $\Gamma(Q)$, so $q_1\sim q_2$.
	
	Next, let $m_1\ge 4$. Let $q_1\coloneqq \pi(p_1)$, $q_2\coloneqq \pi(p_1^{m_1-2})$ and $q_3\coloneqq \pi(p_1^{m_1-2}+p_1^{m_1-1})$. Then $\alpha(p_1)=1$ and $\alpha(q_2)=\alpha(q_3)=m_1-2$, so $q_1,q_2,q_3\in Z(Q)\setminus \{0\}$ by \cref{prop:zero_divisors_quotient_ring}. We further have $q_2\ne q_3$ as $p_1^{m_1}\nmid p_1^{m_1-1}$. By \cref{prop:zero_product_in_quotient}, $q_2$ and $q_3$ are adjacent in $\Gamma(Q)$ because $\alpha(q_2)+\alpha(q_3)=2m_1-4\ge m_1$. In particular, $[q_2]\ne [q_3]$ in $\Gamma_E(Q)$ because $q_2$ is a neighbor of $q_3$, but not of itself. By \cref{prop:zero_product_in_quotient}, $q_1$ is adjacent to neither $q_2$ nor $q_3$ in $\Gamma(Q)$. In particular, $[q_1]$ is distinct from $[q_2]$ and $[q_3]$ and not adjacent to either of them, so $\Gamma_E(Q)$ is not a clique. This implies a lower bound of $1$ on the threshold dimension and boxicity, while \cref{cor:zero_divisor_quotient_threshold} yields an upper bound of $1$.
	\end{proof}
	The remainder of this section is dedicated to proving the following two theorems:
	\begin{theorem}\label{theorem:pid_boxicity}
		Let $n\ge 2$. We have
		\[\boxicity(\Gamma(Q))=\begin{cases}
			n & \text{there is }i\in [n] \text{ with } m_i\ge 3\\
			n & m_i=2 \text{ or } |R/p_iR|\ge 3 \text{ for every }i\in [n]\\
			n-1 & n\ge 3 \text{ and } m_i\in \{1,2\} \text{ for all $i\in[n]$ and there is } i\in [n] \text{ with } m_i=1 \text{ and } |R/p_iR|=2\\
			1 & n=2 \text{ and } m_1,m_2\in\{1,2\}\text{ and there is exactly one } i\in [2] \text{ with } m_i=1 \text{ and } |R/p_iR|=2\\
			0 & n=2 \text{ and } m_1=m_2=1 \text{ and } |R/p_1R|=|R/p_2R|=2
		\end{cases}.\]
	\end{theorem}
	\begin{theorem}\label{theorem:pid_th}
		Let $n\ge 2$. We have
		\[\thdim(\Gamma(Q))=\begin{cases}
			n & \text{there is }i\in [n] \text{ with } m_i\ge 2\\
			n &  |R/p_iR|\ge 3 \text{ for all }i\in[n]\\
			n-1 & n\ge 3 \text{ and } m_i=1 \text{ for all $i\in[n]$ and there is $i\in[n]$ with $|R/p_iR|= 2$ }\\
			1 & n=2 \text{ and } m_1=m_2=1 \text{ and there is exactly one } i\in [2] \text{ with }|R/p_iR|=2\\
			0 & n=2 \text{ and } m_1=m_2=1 \text{ and } |R/p_1R|=|R/p_2R|=2
		\end{cases}.\]
	\end{theorem}
	\subsection{Proof of \cref{theorem:pid_boxicity}}
	The proof of \cref{theorem:pid_boxicity} is organized as follows. \cref{lemma:pid_n_equal_2} provides additional bounds needed for $n=2$. \cref{lemma:pid_some_mi_large,lemma:pid_large_for_every_i} provide lower bounds for the case where the boxicity of $\Gamma(Q)$ is equal to $n$. \cref{lemma:pid_1_2_box_n_1} yields an upper bound for the case where the boxicity of $\Gamma(Q)$ equals $n-1$.
	
	\begin{lemma}\label{lemma:pid_n_equal_2}
		Let $n=2$. If $m_1=m_2=1$ and $|R/p_iR|=2$ for $i\in [2]$, then $\thdim(\Gamma(Q))=\boxicity(\Gamma(Q))=0$. Otherwise, $\thdim(\Gamma(Q))\ge \boxicity(\Gamma(Q))\ge 1$.
	\end{lemma}
	\begin{proof}
		If $m_1=m_2=1$ and $|R/p_iR|=2$ for $i\in [2]$, then $\Gamma(Q)$ contains exactly two vertices $r$ and $s$ satisfying $\alpha_1(r)=1$ and $\alpha_2(r)=0$ and $\alpha_1(s)=0$ and $\alpha_2(s)=1$, respectively. We have $r\cdot s=0$, so $\Gamma(Q)\cong K_2$ has threshold dimension and boxicity $0$. 
		Otherwise, there is $i\in [2]$ with $m_1\ge 2$ or $|R/p_iR|\ge 3$. Assume w.l.o.g.\ that this is the case for $i=1$. If $m_1\ge 2$, then $\pi(p_1)$ and $\pi(p_2)$ are zero divisors that are non-adjacent in $\Gamma(Q)$, implying $\thdim(\Gamma(Q))\ge \boxicity(\Gamma(Q))\ge 1$. If $|R/p_1R|\ge 3$, then there exist $r,s\in R$ with $p_1\nmid r,s,r-s$. Hence, $\pi(r\cdot p_2)$ and $\pi(s\cdot p_2)$ are non-adjacent in $\Gamma(Q)$, implying again  $\thdim(\Gamma(Q))\ge \boxicity(\Gamma(Q))\ge 1$.
	\end{proof}
	\begin{lemma}\label{lemma:pid_some_mi_large}
		Let $n\ge 2$ and suppose that there is $j\in [n]$ with $m_j\ge 3$. Then $\thdim(\Gamma(Q))=\boxicity(\Gamma(Q))=n$.
	\end{lemma}
	\begin{proof}
		By \cref{lemma:simple_bounds_quotient}, it suffices to prove that $\boxicity(\Gamma(Q))\ge n$. We will show this by identifying an $n$-fold join of $\overline{K_2}$, the graph on two vertices without any edge, as an induced subgraph of $\Gamma(Q)$. To this end, for $i\in [n]\setminus \{j\}$, let $x_i\coloneqq \pi(p_j^{m_j-1}\cdot \prod_{\ell\in [n]\setminus \{i,j\}} p_\ell^{m_\ell})$ and $y_i\coloneqq \pi(\prod_{\ell\in [n]\setminus \{i\}}p_\ell^{m_\ell})$. Moreover, define $x_j\coloneqq \pi(p_j\cdot\prod_{\ell\in [n]\setminus \{j\}}p_\ell^{m_\ell})$ and $y_j\coloneqq \pi((p_j+p_j^2)\cdot\prod_{\ell\in [n]\setminus \{j\}}p_\ell^{m_\ell})$. As $p_j^{m_j}\nmid p_j^2$ and by \cref{prop:exponents_well_defined} and \cref{prop:zero_divisors_quotient_ring}, the $x_i$ and $y_i$ constitute pairwise distinct zero divisors. By \cref{prop:zero_product_in_quotient}, we have $x_i\cdot y_i\ne 0$ for $i\in [n]$, but $x_i\cdot x_\ell=x_i\cdot y_\ell=y_i\cdot x_\ell= y_i\cdot y_\ell= 0$ for $i\ne \ell\in [n]$. Hence, $\Gamma(Q)[\{x_i,y_i\colon i\in [n]\}]$ is isomorphic to the $n$-fold join of $\overline{K_2}$.
	\end{proof}
	\begin{lemma}\label{lemma:pid_large_for_every_i}
		Let $n\ge 2$ and suppose that for every $i\in [n]$, $m_i\ge 2$ or $R/p_i R$ contains two non-zero elements. Then $n=\boxicity(\Gamma(Q))=\thdim(\Gamma(Q))$.
	\end{lemma}
	\begin{proof}
		Again, we show how to obtain an induced subgraph of $\Gamma(Q)$ that is isomorphic to the $n$-fold join of $\overline{K_2}$. To this end, for $i\in [n]$ with $m_i\ge 2$, let $x'_i\coloneqq 1$ and $y'_i\coloneqq p_i$, and for $i\in [n]$ with $m_i=1$, let $x'_i,y'_i\in R$ with $p_i\nmid x'_i$, $p_i\nmid y'_i$ and $p_i\nmid x'_i-y'_i$.
		Next, for $i\in [n]$, define $x_i\coloneqq \pi(x'_i\cdot \prod_{j\in [n]\setminus \{i\}} p_j^{m_j})$ and $y_i\coloneqq \pi(y'_i\cdot \prod_{j\in [n]\setminus \{i\}} p_j^{m_j})$. Then the $x_i$ and $y_i$ constitute pairwise distinct zero divisors.
		By \cref{prop:zero_product_in_quotient}, $x_i\cdot y_i\ne 0$, but $x_i\cdot x_\ell=x_i\cdot y_\ell=y_i\cdot x_\ell= y_i\cdot y_\ell= 0$ for $i\ne \ell\in [n]$. Hence, $\Gamma(Q)[\{x_i,y_i\colon i\in [n]\}]$ is isomorphic to the $n$-fold join of $\overline{K_2}$.
	\end{proof}
	\begin{lemma}\label{lemma:pid_1_2_box_n_1}
		Let $n\ge 2$ and suppose that $m_i\in \{1,2\}$ for $i\in [n]$. Suppose further that there is $i\in [n]$ such that $m_i=1$ and $|R/p_iR|=2$. Then we have $\boxicity(\Gamma(Q))\le n-1$. If further $m_i=1$ for all $i\in [n]$, then also $\thdim(\Gamma(Q))\le n-1$.
	\end{lemma}
	\begin{proof}
		We may assume without loss of generality that $|R/p_nR|=2$ and $m_n=1$. We can express $\Gamma(Q)$ as the intersection of $n$ threshold graphs $T(\mu_i,m_i)_{i=1}^n$, where $\mu_i(r)=\alpha_i(r)$. Assume there were $r,s\in Z(Q)\setminus \{0\}$ with $\alpha_i(r)=\alpha_i(s)=m_i$ for $i\in [n-1]$ and $\alpha_n(r)=\alpha_n(s)=0$. Then $p_i^{m_i}\mid r$ and $p_i^{m_i}\mid s$ for $i\in [n-1]$, so also $p_i^{m_i}\mid r-s$. Moreover, both $r$ and $s$ project to the unique element in $R/p_nR\setminus \{0\}$, implying that $p_n\mid r-s$. Hence, $\prod_{i=1}^{n} p_i^{m_i}\mid r-s$, implying $r=s$.
		Hence, we may apply \cref{lemma:box_th_1_2} to conclude the desired statement.
	\end{proof}
	We point out that combining \cref{lemma:pid_some_mi_large,lemma:pid_large_for_every_i,lemma:pid_n_equal_2,lemma:pid_1_2_box_n_1} implies \cref{theorem:pid_boxicity}.
	\subsubsection{Proof of \cref{theorem:pid_th}}
	In this section, we finish the proof of \cref{theorem:pid_th}.
	\cref{lemma:pid_stronger_bound_th} first provides stronger lower bounds to handle the case where $\boxicity(\Gamma(Q))=n-1$, but $\thdim(\Gamma(Q))=n$. Then \cref{lemma:pid_th_equal_box} shows that in all remaining cases, threshold dimension and boxicity of $\Gamma(Q)$ do, in fact, agree.
	\begin{lemma}\label{lemma:pid_stronger_bound_th}
		If $n\ge 2$ and $m_i\ge 2$ for some $i\in [n]$, then $\thdim(\Gamma(R/aR))=n$.
	\end{lemma}
	\begin{proof}
		We first show that $\thdim(\Gamma(R/aR))\le n$.
		Let $G=(V,E)$ be the graph with $V=R/aR$ and $E=\{\{v,w\}\in\binom{V}{2}\colon v\cdot w = 0\}$. Then $\Gamma(R/aR)=G[Z(R/aR)]$ is an induced subgraph of $G$, so it suffices to show that $\thdim(G)\le n$. For $v\in V$ and $i\in [n]$, let $\mu_i(v)\coloneqq \max\{j\in \{0,\dots,m_i\}\colon p_i^j|v\}$. Then for $v,w\in V$, we have $v\cdot w=0$ if and only if for every $i\in [n]$, we have $\mu_i(v)+\mu_i(w)\ge m_i$. But this means that $G$ can be expressed as the intersection of the $n$ threshold graphs $T(\mu_i,m_i)_{i\in [n]}$.
		Next, we prove the lower bound. For ease of notation, we will assume without loss of generality that $m_n\ge 2$. For $i\in [n]$, let $G_i=(\{x_i,y_i\},\emptyset)$ consist of two isolated vertices. We will show that $\Gamma(R/aR)$ contains a partial join of $(G_i)_{i\in [n]}$ as an induced subgraph. As $\thdim(G_i)=1$ for $i\in [n]$, this will conclude the proof. For $i\in [n-1]$, we define $v_i\coloneqq\sigma(x_i)\coloneqq \prod_{j\in [n]\setminus \{i\}} p_j^{m_j}$ and $w_i\coloneqq \sigma(y_i)\coloneqq \prod_{j\in [n-1]\setminus \{i\}} p_j^{m_j}\cdot p_n^{m_n-1}$.
		We further let $v_n\coloneqq \sigma(x_n)\coloneqq \prod_{j\in [n-1]} p_j^{m_j}$ and $w_n\coloneqq \sigma(y_n)\coloneqq \prod_{j\in [n-1]} p_j^{m_j}\cdot p_n^{m_n-1}$.
		\begin{claim}
			The ring elements $v_i,w_i,i\in [n]$ constitute pairwise distinct zero divisors.
		\end{claim}
		\begin{proof}[Proof of claim.]
			For $i\in [n-1]$, $v_i\ne w_i$ because $p_n^{m_n}\mid v_i$, but $p_n^{m_n}\nmid w_i$.
			We have $p_n\mid v_i,w_i$, but $p_i\nmid v_i,w_i$, so both are zero divisors.
			We have $v_n\ne w_n$ because $p_n\mid w_n$ but $p_n\nmid v_n$. As $p_n^{m_n}\nmid v_n,w_n$ but $p_1\mid v_n,w_n$, both are zero divisors.
			Finally, $\{v_i,w_i\}\cap \{v_j,w_j\}=\emptyset$ for $1\le i < j\le n$ because $p_i\nmid v_i,w_i$, but $p_i\mid v_j,w_j$.
		\end{proof}
		\begin{claim}
			$\Gamma(R/aR)[\{v_i,w_i\colon i\in [n]\}]$ constitutes a partial join of $(G_i)_{i\in [n]}$.
		\end{claim}
		\begin{proof}[Proof of claim]
			For every $i\in [n]$, we have $p_i^{m_i}\nmid v_i\cdot w_i$, so $\sigma$ induces an isomorphism between $G_i$ and $\Gamma(R/aR)[\{v_i,w_i\}]$. Next, let $1\le i< j\le n$.
			Then $v_i\cdot v_j=0$ because $p_k^{m_k}\mid v_i$ for all $k\in [n]\setminus \{i\}$ and $p_i^{m_i}\mid v_j$. We further have $w_i\cdot w_j=0$ because $p_k^{m_k}\mid w_i$ for all $k\in [n-1]\setminus \{i\}$, $p_i^{m_i}\mid w_j$, and $p_n^{m_n-1}\mid w_i,w_j$. Hence, $M_{i,j}\coloneqq \{\{a,b\}\colon a\in \{v_i,w_i\},b\in\{v_j,w_j\},\{a,b\}\notin E(\Gamma(R/aR))\}\subseteq \{\{v_i,w_j\},\{v_j,w_i\}\}$, which is a matching.
		\end{proof}
	\end{proof}
	\begin{lemma}\label{lemma:pid_th_equal_box}
		Let $n\ge 2$ and let $m_i=1$ for all $i\in [n]$. Then \[\thdim(\Gamma(R/aR))=\begin{cases}
			0 & n=2 \text{ and } |R/p_1R|=|R/p_2R|=2\\
			1& n=2 \text{ and } |R/p_iR|=2 \text{ for exactly one }i\in[2]\\
			n-1 & n\ge 3 \text{ and } |R/p_iR|=2 \text{ for some }i\in[n]\\
			n & |R/p_iR|\ge 3 \text{ for all }i\in [n]
		\end{cases}.\]
	\end{lemma}
	\begin{proof}
		If $n=2$ and $|R/p_1R|=|R/p_2R|=2$, the desired result follows from \cref{lemma:pid_n_equal_2}. In all remaining cases, \cref{lemma:simple_bounds_quotient,lemma:pid_n_equal_2,lemma:pid_large_for_every_i,lemma:pid_1_2_box_n_1} yield matching upper and lower bounds of $n-1$.
	\end{proof}
	Combining \cref{lemma:pid_stronger_bound_th,lemma:pid_th_equal_box} proves \cref{theorem:pid_th}.

	\bibliographystyle{plainurl}
	
	\bibliography{references}

\end{document}